\documentclass[12pt]{extarticle}
\usepackage{amsmath, amsthm, amssymb, hyperref, color}
\usepackage[linesnumbered,commentsnumbered,ruled,vlined]{algorithm2e}
\usepackage{graphicx}
\usepackage{caption}
\usepackage{subcaption}
\usepackage{mathtools}
\usepackage{enumerate}
\usepackage[all]{xypic}
\usepackage{verbatim}
\usepackage{upgreek}
\usepackage{tikz}
\usepackage{caption}
\oddsidemargin \evensidemargin
\theoremstyle{theorem}
\newtheorem{theorem}{Theorem}
\newtheorem{theoremstar}[theorem]{Theorem*}
\newtheorem{proposition}[theorem]{Proposition}
\newtheorem{lemma}[theorem]{Lemma}
\newtheorem{corollary}[theorem]{Corollary}
\theoremstyle{definition}

\newtheorem{remark}[theorem]{Remark}

\newtheorem{examplex}[theorem]{Example}
\newenvironment{example}
  {\pushQED{\qed}\examplex}
  {\popQED\endexamplex}

\numberwithin{theorem}{section}

\newcommand{\PP}{\mathbb{P}}
\newcommand{\RR}{\mathbb{R}}
\newcommand{\QQ}{\mathbb{Q}}
\newcommand{\CC}{\mathbb{C} }
\newcommand{\ZZ}{\mathbb{Z}}
\newcommand{\NN}{\mathbb{N}}

\newcommand\sbullet[1][.5]{\mathbin{\vcenter{\hbox{%
					\scalebox{#1}{$\bullet$}}}}%
}

\newcommand{\rvline}{\hspace*{-\arraycolsep}\vline\hspace*{-\arraycolsep}}

\usepackage{mathdots}

\title{\bf Likelihood Degenerations}

\author{Daniele Agostini, Taylor Brysiewicz, Claudia Fevola \\
  Lukas K\"uhne, Bernd Sturmfels, and Simon Telen \medskip \\
 {(with an appendix by Thomas Lam)}}

\date{}

\begin{document}

\maketitle
\begin{abstract}
\noindent Computing all critical points of a monomial
on a very affine variety is a fundamental task in algebraic
statistics, particle physics and other fields. The number of critical points is
known as the maximum likelihood (ML) degree. When the  variety 
is smooth, it coincides with the Euler characteristic.
We introduce degeneration techniques
that are inspired by the soft limits in CEGM theory, and we
answer several questions raised in the physics literature.
These pertain to bounded regions in discriminantal arrangements
and to moduli spaces of point configurations. 
We present theory and practise, connecting complex geometry, tropical combinatorics,
 and numerical nonlinear algebra.
\end{abstract}

\section{Introduction} \label{section1}

A {\em very affine variety} $\,X$ is a closed subvariety of an algebraic  torus $(\CC^*)^n$. 
For any integer vector ${\bf u} = (u_1,\ldots,u_n) \in \ZZ^n$, the  Laurent monomial
${\bf z}^{\bf u} = z_1^{u_1} \cdots z_n^{u_n}$ is a regular function on $(\CC^*)^n$, 
and we are interested in the set of critical points of ${\bf z}^{\bf u}$ on~$X$.
The natural approach is via the gradient of the {\em log-likelihood function}
${\rm log}({\bf z}^{\bf u}) = \sum_{i=1}^n u_i {\rm log}(z_i)$. This makes
sense for any complex vector ${\bf u} \in \CC^n$. The coordinates of
$\nabla {\rm log}({\bf z}^{\bf u})$ are rational functions, and we seek points ${\bf z} \in X$
at which that gradient vector lies in the normal space. This leads to
a system of rational function equations whose solutions are the critical points.
 Their number  is independent of ${\bf u}$, provided ${\bf u}$ is generic.
 This is an invariant of $X$, denoted 
 ${\rm MLdegree}(X)$, and known as the {\em maximum likelihood degree}; 
 see \cite{CHKS, GR, Huh13, HS}. Whenever $X$ is smooth,  
 we know from \cite[Theorem 1]{Huh13} that it coincides with the signed Euler characteristic of $X$:
\begin{equation}\label{eq:mldegreeeulchar}
\operatorname{MLdegree}(X) \,\,=\,\, (-1)^{\dim X} \cdot \chi (X).
\end{equation} 

The term \emph{likelihood} comes from statistics \cite{PS}. A discrete statistical model
on $n-1$ states is a subset $M$ of the probability simplex 
$\Delta_{n-2} = \{ (z_1,\ldots,z_{n-1}) \in \RR^{n-1}_{> 0} : z_1 + \cdots + z_{n-1} = 1 \}$.
In algebraic statistics, the model $M$ is a semialgebraic set, and one replaces
$M$ by its Zariski closure $X$. That closure is taken in the  torus $(\CC^*)^n$,
where $z_n = z_1 + \cdots + z_{n-1} $. After collecting data, we write
$u_i \in \NN$ for the number of samples in state $i$. The 
monomial ${\bf z}^{\bf u}$ is the likelihood function. The
goal of likelihood inference is to maximize ${\bf z}^{\bf u}$ over $M$.
Thus, statisticians seek those critical points of ${\rm log}({\bf z}^{\bf u})$
that are real and positive. The ML degree of $X$ is an algebraic
complexity measure for performing likelihood inference with the model~$M$.

In particle physics, very affine varieties arise from \emph{scattering equations} 
\cite{CEGM, CHY, CUZ, ST}.
These equations are central to
the Cachazo--He--Yuan (CHY) formulas for biadjoint scalar amplitudes. Their solutions are the critical points of $\log({\bf z}^{\bf u})$ on the moduli space 
$\mathcal{M}_{0,m} = \text{Gr}(2,m)^\circ/(\CC^*)^m$ of genus zero curves with $m$ marked points. In the more general context of Cachazo--Early--Guevara--Mizera (CEGM) amplitudes \cite{CEGM}, one considers the very affine varieties $X(k,m) = \text{Gr}(k,m)^\circ/(\CC^*)^m$ with $k \geq 2$. These are moduli spaces of $m$ points in $\mathbb{P}^{k-1}$ in linearly general position, natural generalizations of ${\cal M}_{0,m}$.  While the maximum likelihood degree ${\rm MLdegree}(X(2,m))$ is known to be $(m-3)!$, much less is known for $k \geq 3$.
The connection between maximum likelihood and scattering equations was developed in
\cite{ST}.

The term \emph{degeneration} comes from algebraic geometry. It represents the idea of studying the properties of a general object $X_t$ for $t\ne 0$ by letting it degenerate to a more special object $X_0$, which is often easier to understand. This corresponds to finding a nice compactification of the variety  $\mathcal{X}^0 = \cup_{t\ne 0} X_t$ to a variety $\mathcal{X}$, by adding the special fiber $X_0$. There are many possibilities for such constructions. Most relevant in our context are the tropical compactifications in
\cite[\S 6.4]{MS} and their connection to likelihood inference in~\cite{VS}.

Cachazo, Umbert and Zhang \cite{CUZ} introduced a class of degenerations
called {\em soft limits}. 
The present article arose from our desire to gain
a mathematical understanding of that construction from physics.
We succeeded in reaching that understanding, and we here share it
from multiple perspectives:
algebraic geometry, combinatorics and numerical mathematics.

We now give an overview of our contributions.
We start in Section \ref{section2} with a discussion of the
moduli space $X(k,m)$, and how it fits into the
framework of very affine varieties. 
The soft limits in \cite{CUZ} are special instances 
 of likelihood degenerations that are well adapted to the geometry of configurations.
We explain how these are related to the
deletion maps 
\begin{equation}
\label{eq:deletion}
\pi_{k,m}\,:\, X(k,m+1)\, \rightarrow \,X(k,m).
\end{equation}
 These maps are
shown to be stratified fibrations.
We discuss both the strata and the fibers.
This sets the stage for
the computation of Euler characteristics by combinatorial methods.

The fibers of (\ref{eq:deletion}) are complements of
discriminantal hyperplane arrangements.
By Varchenko's Theorem \cite[Theorem 3]{CHKS},
the ML degree is the number of bounded regions.
 In Section \ref{section3}
we focus on the generic fiber. This arises from
the $\binom{m}{k-1}$ hyperplanes spanned by $m$ generic points in $\PP^{k-1}$.
We show that, for fixed $k$, the number of bounded regions is a polynomial in $m$
of degree $(k-1)^2$. This polynomial was denoted
${\rm Soft}_{k,m}$ in \cite{CUZ}. We display it  explicitly for  $k \leq 7$.
Our result extends to all coefficients of the characteristic polynomial
(Theorem \ref{thm:arepolynomials}). Its proof rests on constructions of
Koizumi, Numata and Takemura in~\cite{Koizumi}.

In Section \ref{section4} we turn from fibers to strata in the base
of (\ref{eq:deletion}). These can be modeled as
strict  realization spaces of matroids.
Such matroid strata are very affine varieties.
In  Theorem~\ref{thm:matroid_ml} we furnish a comprehensive study
for small matroids 
of  rank $k$ on $m$ elements. Note that the uniform
matroid corresponds to $X(k,m)$.
We compute the ML degrees for  all matroids in the range
 $k=3, m \leq 9$ and $k=4,m=8$.
 For larger matroids this would become infeasible by
{\em Mn\"ev's Universality Theorem}
\cite[\S 6.3]{BokStu}.
Our result is achieved by integrating
software tools from 
computer algebra \cite{ZariskiFrames}, combinatorics  \cite{MMIB_DB},
and certified numerics \cite{BRT, HCjl}.

In Section \ref{section5} we examine the space
$X(3,m)$ of $m$ points in general position in the projective plane $\PP^2$.
Cachazo,  Umbert and Zhang \cite{CUZ} report that 
the  ML degree of $X(3,m)$ equals
$ 26$ for $m=6$, $ \,1\,272$  for $m=7$, and $188\,112$ for $m=8$.
These numbers are denoted $\mathcal{N}^{(3)}_{m}$ in~\cite{CUZ}.
Thomas Lam (Appendix \ref{appendixA}) derived them using finite field methods, and he also computed
$$ \mathcal{N}^{(3)}_{9} \,\, = \,\, {\rm MLdegree}(X(3,9))  \,\,  =  \,\, 74\, 570\, 400. $$
We present a topological proof
of these results, and we prove the conjecture in \cite[\S 6]{CUZ}.
This involves a careful study of the
stratified fibration (\ref{eq:deletion}).
The combinatorics we develop along the way,
such as posets of strata and M\"obius functions,
 should be of independent interest.

Section \ref{section6} is dedicated to configurations of eight points in projective $3$-space.
Based on our computational results, we predict that the ML degree of $X(4,8)$ is equal to
$5\,211\,816$. This is the number of solutions to the likelihood equations,
found numerically by the software {\tt HomotopyContinuation.jl} \cite{HCjl}.
We present a detailed analysis of the tropical geometry of
 soft limits in this case. This confirms the combinatorial predictions made in \cite[Table~2]{CUZ},
 and offers a blueprint for future research that connects tropical geometry and numerical analysis.

In Section \ref{section7} we turn to algebraic statistics, and we follow up on earlier
work on likelihood degenerations due to Gross and Rodriguez \cite{GR}.
We introduce the tropical version of maximum likelihood estimation for
discrete statistical models. This arises by replacing
the real numbers $\RR$ by the Puiseux series field $\RR\{\! \{t \} \! \}$,
in both the data and the solutions.
Our main result (Theorem \ref{thm:tropMLE}) characterizes
the tropical MLE for linear discrete models. This involves the
intersection of two Bergman fans, corresponding to a dual pair of matroids.

In Section~\ref{section8} we present numerical methods for likelihood degenerations,
with an emphasis on recovering the description of tropical curves  from floating point coordinates.
This extends the tropical MLE approach in Section~\ref{section7} from linear models to other very affine
varieties, and it allows us to find tropical solutions to scattering equations in particle physics.

The paper concludes with Appendix \ref{appendixA}, written by Thomas Lam,
which gives the computation of ML degrees for $k=3$ using finite field methods,
based on the Weil conjectures.

The code used in this paper together with computational results are available~at
\[
	\texttt{\href{https://mathrepo.mis.mpg.de/LikelihoodDegenerations}{https://mathrepo.mis.mpg.de/LikelihoodDegenerations}}.
\]

\section{Very affine varieties of point configurations} \label{section2}

This work revolves around a very affine variety arising in particle physics
\cite{CEGM, CHY, CUZ}, namely
  the moduli space $X(k,m)$ of $m$ points in $\mathbb{P}^{k-1}$ in linearly general position. More explicitly, this moduli space parametrizes $m$-tuples $[P]=[P_1,\dots,P_m]$ of points $P_i\in \mathbb{P}^{k-1}$ such that no $k$ of them lie on a
 hyperplane.    Moreover, the $m$-tuples are considered only up to the action of
    ${\rm GL}(k,\mathbb{C})$.     The dimension of $X(k,m)$ equals $(k-1)(m-k-1)$.
        Observe that $X(2,m)$ is  the familiar
        $(m-3)$-dimensional moduli space $\mathcal{M}_{0,m}$ of $m$ distinct labeled points on $\mathbb{P}^1$.

We next present a parametrization which shows that $X(k,m)$ is very affine.
Taking homogeneous coordinates for the points $P_i$, any $m$-tuple $[P]$ as above can be represented by a 
complex $k\times m$ matrix whose $k \times k$ minors are all nonzero. Two such representations are equivalent if and only if they differ by left multiplication by
${\rm GL}(k,\mathbb{C})$, or by a rescaling of the columns by an element in the torus $(\mathbb{C}^*)^m$. Hence, this identifies $X(k,m)$ as the quotient 
\begin{equation} \label{eq:asthequotient} X(k,m) \,= \, {\rm Gr}(k,m)^\circ/(\mathbb{C}^*)^m,
\end{equation}
where ${\rm Gr}(k,m)$ is the Grassmannian of $k$-dimensional subspaces in $\CC^m$ and
${\rm Gr}(k,m)^\circ$ is the open cell where all Pl\"ucker coordinates $p_{i_1 \dots i_k}$ are nonzero. 
We can uniquely write the homogeneous coordinates of an $m$-tuple $[P]\in X(k,m)$ as the columns of the 
$k \times m$ matrix
\begin{equation}\label{eq:generalmatrix}
M_{k,m} \,\,=\,\,  \begin{bmatrix} 
0 & 0 & 0 & \dots & 0 & (-1)^k & 1 & 1 & 1 & \dots & 1\\ 
0 & 0 & 0 & \dots & (-1)^{k-1} & 0 & 1 & x_{1,1} & x_{1,2} & \dots & x_{1,m-k-1} \\ 
\vdots & \vdots & \vdots & \iddots & \vdots & \vdots & \vdots & \vdots & \vdots & \iddots & \vdots\\ 
0 & 0 & -1 & \dots & 0 & 0 & 1 & x_{k-3,1} & x_{k-3,2} & \dots & x_{k-3,m-k-1}\\ 
0 & 1 & 0 & \dots & 0 & 0 & 1 & x_{k-2,1} & x_{k-2,2} & \dots & x_{k-2,m-k-1}\\ 
-1 & 0 & 0 & \dots & 0 & 0 & 1 & x_{k-1,1} & x_{k-1,2} & \dots & x_{k-1,m-k-1}\\
\end{bmatrix} , 
\end{equation}
provided all maximal minors $p_{i_1 \dots i_k}$ of this matrix are nonzero. 
The antidiagonal matrix in the left $k \times k$ block was chosen so that
each unknown $x_{i,j}$ equals such a minor for $i_1 < \cdots < i_k$.
This identifies $X(k,m)$ as an open subset of the torus $(\CC^*)^{(k-1)(m-k-1)}$
with coordinates $x_{i,j}$.
 The Pl\"ucker embedding realizes $X(k,m)$ as a closed subvariety of a high-dimensional torus:
 \begin{equation}\label{eq:torusembedding} 
X(k,m) \hookrightarrow (\mathbb{C}^*)^{\binom{m}{k}}, \qquad M_{k,m} \mapsto \left( p_{i_1\dots i_k} \right) .
\end{equation}
The corresponding log-likelihood function, known as the \emph{potential function} in physics, is
\begin{equation}\label{eq:potentialfunction} 
\mathcal{L}_{k,m} \,\,=\,\, \sum_{i_1\dots i_k} u_{i_1\dots i_k} \cdot  \log( p_{i_1\dots i_k} ).
\end{equation}
The critical point equations, known as \emph{scattering equations} in physics, are given by
\begin{equation}\label{eq:scattering} 
\qquad \frac{\partial \mathcal{L}_{k,m}}{ \partial x_{i,j}} \,=\, 0 \qquad \text{ for }\, 1\leq i \leq k-1\, \,{\rm and} \,\, 1\leq j \leq m-k-1 .
\end{equation}
This is a system of rational function equations.
One way to compute the ML degree of $X(k,m)$ is by counting the number of solutions 
 for general values of the parameters $u_{i_1\dots i_k}$. This can be done by finding explicit
 solutions to (\ref{eq:scattering}), notably by the numerical approach described
 in \cite[\S 3]{ST} which rests on the software {\tt HomotopyContinuation.jl}~\cite{HCjl}.

\begin{remark} 
{\em Likelihood degenerations} arise when the coefficients $u_{i_1 \cdots i_k}$
depend on a parameter $t$, and one studies the behavior of the
solutions in the limit $t \rightarrow 0$. A special instance
are the soft limits of Cachazo et al.~\cite{CUZ}. This likelihood degeneration
is presented in (\ref{eq:L48withparameter}). We explain in Section~\ref{section6}
how soft limits  are related to  the methods in this section.
Sections \ref{section7} and \ref{section8} explore
arbitrary likelihood degenerations, with a view towards statistics and numerics.
\end{remark}

Another approach to computing the ML degree is topological, through the Euler characteristic. 
Indeed, since the variety $X(k,m)$ is smooth,
the formula \eqref{eq:mldegreeeulchar} applies and gives
\begin{equation}\label{eq:eulerxkn}
 \operatorname{MLdegree}( X(k,m)) \,=\, (-1)^{(k-1)(m-k-1)}\cdot \chi(X(k,m)).
\end{equation} 

Our strategy to compute the Euler characteristic is to exploit the  deletion map in (\ref{eq:deletion}):
\[ \pi_{k,m}\colon X(k,m+1) \longrightarrow X(k,m), \qquad [P_1,\dots,P_m,P_{m+1}] \mapsto [P_1,\dots,P_{m}] .\]
This approach relies on the fact that the Euler characteristic is multiplicative along fibrations. 

\begin{example}[$k=2$]
	The space $X(2,3)$ is a single point, so $\chi(X(2,3))=1$. For $m\geq 3$ we consider the map $\pi_{2,m}\colon X(2,m+1) \longrightarrow X(2,m)$. A point in $X(2,m)$ is identified with an  $m$-tuple $\left[P_1,\dots,P_{m}\right]$, where  $P_i\in \mathbb{P}^1$ are pairwise distinct, and  $P_1=(0:-1), P_2=(1:0),P_3 = (1:1)$ are fixed. The fiber over this point equals $F = \mathbb{P}^1\setminus \{ P_1,\dots,P_m \}$. This has Euler characteristic $2-m$. 
	All fibers are homeomorphic to $F$, so the map $\pi_{2,m}$ is a fibration.
	Multiplicativity of the Euler characteristic along fibrations implies
	 $\chi(X(2,m+1)) = \chi(F)\cdot\chi(X(2,m))= (2-m)\cdot \chi(X(2,m))$.
	 By induction on $m$, we conclude that $\chi(X(2,m)) = (-1)^{m-3}(m-3)!$. 
\end{example}

Our main difficulty for $k \geq 3$ is that the deletion map $\pi_{k,m}$ is generally not a fibration.
But it is a {\em stratified fibration}, that is,
it is a map $f\colon X\to Y$ of complex algebraic varieties such that $Y$ has	
a  stratification $\mathcal{S} = \{ S \subseteq Y \}$ by finitely many closed strata, with the property that, over each open stratum $S^\circ = S\setminus \bigcup_{S'\subsetneq S} S'$, the map is a fibration with fiber~$F_S$.
 
The set $\mathcal{S}$ of all strata $S$ in a stratified fibration  $f\colon X\to Y$  is naturally a poset,
  ordered by inclusion. We can use this combinatorial structure to compute the Euler characteristic of $X$.
 The following result is standard but we include it here for the sake of reference. 

\begin{lemma}\label{lemma:stratifiedfibration}
	Let $f\colon X \to Y$ be a stratified fibration and $\mu$ the M\"obius function of  $\mathcal{S}$.
	 Then 
	\begin{align*}
	\chi(X) \,\,&= \,\,\, \sum_{S\in \mathcal{S}} \chi(S) \,\cdot \! \!\!\!\sum_{S'\in \mathcal{S},S'\supseteq S} \!
	\!\! \mu(S,S')\chi(F_{S'}) \,\,\,\\ 
	&= \,\,\,
	\chi(Y)\cdot \chi(F_Y)\,+\,\sum_{S\in \mathcal{S}} \chi(S) \,\cdot \!\!\!\!\!\sum_{S'\in \mathcal{S},S'\supseteq S} \!\! \mu(S,S')\cdot(\chi(F_{S'}) - \chi(F_Y) ).
	\end{align*}
\end{lemma}

\begin{proof}
By the excision property of the Euler characteristic, together with the multiplicativity along fibrations, we know
that $\chi(X) = \sum_{S'\in  \mathcal{S}} \chi(S'^\circ)\cdot \chi(F_{S'})$.
We can rewrite this as follows: for any closed stratum $S'\in \mathcal{S}$ we have $\chi(S') = \sum_{S'\supseteq S} \chi(S^\circ)$. The M\"obius inversion formula yields $\chi(S'^\circ) = \sum_{S'\supseteq S} \mu(S,S')\chi(S)$.
 Plugging this into the previous formula, we obtain the first equality in Lemma \ref{lemma:stratifiedfibration}. The second equality comes 
 from the definition of the M\"obius function, which stipulates that
$\sum_{S'\supseteq S}\mu(S,S') = 0$, for a fixed stratum $S\in \mathcal{S}\setminus \{Y\}$.
\end{proof}

Lemma \ref{lemma:stratifiedfibration}
 can be used to compute the ML degree of a very affine variety $X$ that is smooth. If we are
given a stratified fibration $X \rightarrow Y$, then the computation
 reduces to the topological task of computing the Euler characteristics of the fibers and of the strata, together with the combinatorial task of computing the M\"obius function of the poset.

We shall apply this method to the deletion map in (\ref{eq:deletion}). For this, we need to argue
that this map is a stratified fibration, which requires us to identify the strata and the fibers. The fibers are
 complements of discriminantal arrangements, to be discussed thoroughly in Section \ref{section3}.
 The strata are given by a certain matroidal stratification, see Section \ref{section4}. We describe the general setting here, and we will give explicit computations for the case $k=3$ in Section \ref{section5}.

First, let us consider the fibers. Given a representative $[P] = [P_1,\dots,P_m]$ for an element  of $X(k,m)$, the corresponding
discriminantal arrangement is the set $\mathcal{B}(P)$ of all hyperplanes linearly spanned by any set of $k-1$ points $P_{i_1},\dots,P_{i_{k-1}}$. Observe that these span a hyperplane because of the requirement 
that the points are in linearly general position. The fiber of $\pi_{k,m}$ over  $[P]$ is 
 the complement  of this  hyperplane arrangement:
 \begin{equation}
 \label{eq:fiberB}
 \pi_{k,m}^{-1}([P]) \,\,\,\cong \,\,\,\mathbb{P}^{k-1} \setminus \!\! \bigcup_{H\in \mathcal{B}(P)} H .
 \end{equation}
The Euler characteristic of such a complement is found with the methods in Section~\ref{section3}.
One subtle aspect of this computation is that all objects are taken up to the action of~${\rm GL}(k, \CC)$.

This description also indicates the appropriate stratification of $X(k,m)$: it depends on the linear dependencies satisfied by the $\binom{m}{k-1}$ hyperplanes in the arrangement $\mathcal{B}(P)$. This comes from a certain matroid stratification.
 Let us give an example for the case $k=3$:

\begin{example}\label{example:X36} Consider $\pi_{3,6}\colon X(3,7) \to X(3,6)$. To any  $[P]$
in $ X(3,6)$ we  associate the  arrangement  of $15$ lines $\overline{P_i P_j}$.
Move one of them to the line at infinity in $\PP^2$. For generic~$[P]$, the number of bounded regions
in the resulting  arrangement of $14$ lines is the Euler characteristic of the fiber
 $\pi_{3,6}^{-1}([P])$. This number is $42$, as we shall see in Section~\ref{section3}.
  For special fibers the Euler characteristic drops. Consider $[P] $ where the
    lines $ \overline{P_1 P_2}, \,\overline{P_3 P_4},\,\overline{P_5 P_6}$ are concurrent.
  These three lines are not in  general position \cite[Figure 4-3]{BokStu}.
      The Euler characteristic of the fiber of $\pi_{3,6}$ over $[P]$ is $41$.
  Our stratification of $X(3,6)$ must account for this.
  \end{example}

In general, the situation is as follows.  Each point of $X(k,m)$ is represented by a
$k \times m$ matrix $M_{k,m}$ as in  \eqref{eq:generalmatrix}. We consider  
the $(k-1)$st exterior power $\wedge_{k-1} M_{k,m}$ of that matrix. That new matrix
has $k$ rows and $\binom{m}{k-1}$ columns, and its entries are the signed
$(k-1) \times (k-1)$ minors of $M_{k,m}$. Each column of $\wedge_{k-1} M_{k,m}$ is the normal vector
to the hyperplane spanned by $k-1$ points $P_{i_1},\ldots,P_{i_{k-1}}$.
Here we allow for the possibility that the column vector is zero, which means
that $P_{i_1},\ldots,P_{i_{k-1}}$ lie on a $(k-2)$-plane in $\PP^{k-1}$.
Taking the exterior power is equivariant with respect to the action of ${\rm GL}(k,\CC)$,
so we obtain a map of Grassmannians
\begin{equation}
\label{eq:mapgrass}
\begin{matrix}
 {\rm Gr}(k,m)\, \rightarrow \,{\rm Gr}\bigl(k,\binom{m}{k-1}\bigr). \end{matrix}
 \end{equation}
 
 We now consider the matroid stratification \cite[\S 4.4]{BokStu} of the big Grassmannian
$ {\rm Gr}\bigl(k,\binom{m}{k-1}\bigr)$. The strata correspond to arrangements
of $\binom{m}{k-1}$ hyperplanes in $\PP^{k-1}$ that have a fixed intersection lattice.
We next  consider the  pullback of this matroid stratification under the map (\ref{eq:mapgrass}).
This is a very fine stratification of ${\rm Gr}(k,m)$. Its strata correspond to point
configurations whose discriminantal hyperplane arrangement has a fixed intersection lattice.
This stratification is much finer than the matroid stratification of ${\rm Gr}(k,m)$. In particular,
it defines a highly nontrivial stratification of the open cell ${\rm Gr}(k,m)^\circ$. This stratification of ${\rm Gr}(k,m)^\circ$ is
compatible with the action of the torus $(\CC^*)^m$, given that all our constructions are based on
projective geometry. Passing to the quotient in
(\ref{eq:asthequotient}), let $\mathcal{S}_{k,m}$ denote the
induced stratification of $X(k,m)$.
We call this the {\em discriminantal stratification} of  the space $X(k,m)$.

\begin{proposition}\label{prop:fibration}
	The deletion map $\pi_{k,m}\colon X(k,m+1) \to X(k,m)$ defines a stratified fibration with respect to the 
	discriminantal stratification on the very affine variety $X(k,m)$.
\end{proposition}

\begin{proof}[Proof and discussion]
Each fiber of $\pi_{k,m}$ is the complement of a discriminantal hyperplane arrangement.
These arrangements have fixed intersection lattice as the base point $[P]$ ranges
over an open stratum of $\mathcal{S}_{k,m}$. This implies that the Euler characteristic is constant
on each fiber. This is what we need to be a fibration on each stratum.  It is possible that the homotopy type
varies across such fibers \cite{Ryb}. This would require further subdivisions into constructible sets.
But, all we need here is for the Euler characteristic to be constant.
\end{proof}

In conclusion, the topological approach to computing the ML degree of $X(k,m)$
consists of two parts: computing the Euler characteristic of the fibers, which is done in Section \ref{section3}, and  computing the Euler characteristic of the matroid strata, which is done in Section \ref{section4}.

\section{Discriminantal hyperplane arrangements} \label{section3}

This section concerns the discriminantal hyperplane arrangements that form the fibers (\ref{eq:fiberB}).
Our main focus lies on the generic fibers. Such arrangements have been studied 
 for decades, e.g.~in \cite{Crapo,Falk,Koizumi,OT}. Our general reference on the relevant combinatorics is \cite{Stanley}.

We work with two models in real affine space. First, fix $m$ general points in $\RR^{k-1}$. We write
$\mathcal{A}(k,m)$ for the arrangement of $\binom{m}{k-1}$ hyperplanes spanned any
$k-1$ of these points. Second, we consider the set $\widetilde{\mathcal{B}}(k,m)$ of
 $\binom{m}{k-1}$ hyperplanes through the origin in $\RR^k$
that are spanned by any $k-1$ of the columns of the  matrix $M_{k,m}$ in (\ref{eq:generalmatrix}),
where the $x_{i,j}$ are generic. Let $\mathcal{B}(k,m)$ be the restriction of this hyperplane
arrangement to $\RR^{k-1} \simeq \{x_1 = 1\}$. Thus $\mathcal{B}(k,m)$ is an affine arrangement
of $\binom{m}{k-1}-1$ hyperplanes in $\RR^{k-1}$.
As pointed out by Falk in~\cite{Falk}, the combinatorics of the arrangement  $\mathcal{A}(k,m)$ is
not independent of the choice of the $m$ general points as for some choices the dependencies among
those points admit ``second-order'' dependencies beyond the Grassmann-Pl\"ucker relations.
Following Bayer and Brandt~\cite{BB} we assume that the $m$ points are \emph{very generic} in the sense that such dependencies do not appear.
All such very generic choices yield combinatorially equivalent arrangements $\mathcal{A}(k,m)$, see also~\cite{SY} for a discussion of non-very generic discriminantal arrangements.

In Section~\ref{section5} we will explore non-very generic degenerations of these discriminatal arrangements.
Both $\mathcal{A}(k,m)$ and $\mathcal{B}(k,m)$
are induced by the open cell in the large Grassmannian under
an embedding of the form (\ref{eq:mapgrass}). 
But their combinatorial structures are different. 

\begin{figure}[h]
		\vspace{-0.06cm}
		\begin{center}
			\includegraphics[scale=0.28]{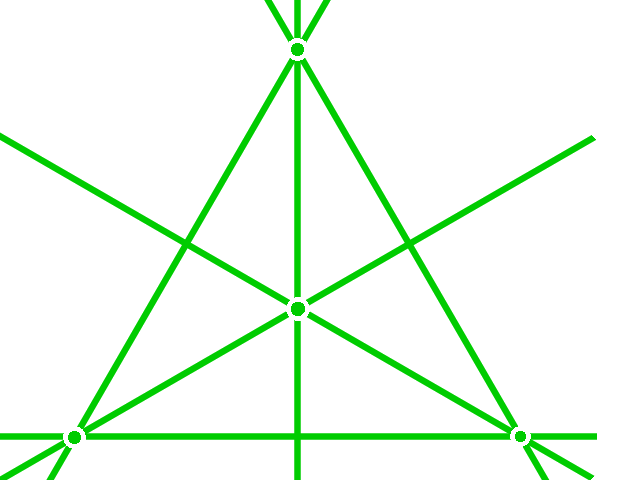} 
			\qquad \qquad
			\includegraphics[scale=0.28]{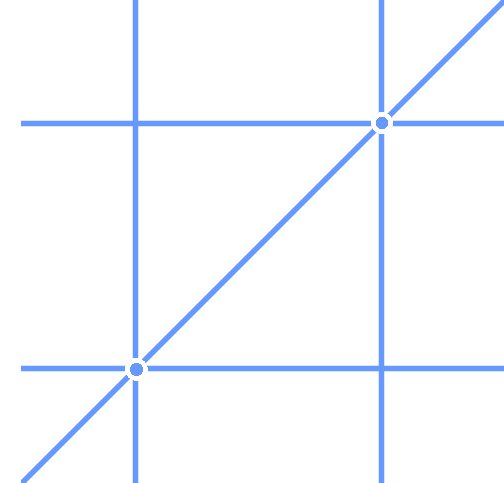} 
		\end{center}
		\vspace{-0.4cm}
		\caption{The discriminantal line arrangements $\mathcal A(3,4)$ (left) and $\mathcal B(3,4)$ (right). 
		\label{fig:twoarrangements}}
	\end{figure}

\begin{example}[$k=3,m=4$]  \label{ex:A34greenblue}
The arrangement $\mathcal{A}(3,4)$ consists of the six lines that are spanned
by four general points in the affine plane $\RR^2$. Its complement in $\RR^2$ consists of~$18$ regions, of which six are bounded. The arrangement $\mathcal{B}(3,4)$ is
obtained from $\mathcal{A}(3,4)$ by moving one of the six lines to infinity. Thus
$\mathcal{B}(3,4)$ consists of five lines, which divide $\RR^2$ into $12$ regions, of which  two are bounded.
The two arrangements are shown in Figure \ref{fig:twoarrangements}.
Their characteristic polynomials \eqref{eq:char_poly} are $\chi_{\mathcal A(3,4)}(t) = t^2-6t+11$ and $\chi_{\mathcal B(3,4)}(t) = t^2 -5t+6$.
\end{example}
 
\begin{table}[h]
\begin{small}
		\[\def\arraystretch{1}{
			\begin{array}{c| r r r r r r r r r r}			
			 & \,m\,=\, 2 & \,\,3 &  \,4 & 5 & 6 & 7 & 8 & 9 \\
			\hline
			k=2    & 0 & 1 &  2 & 3   &4 & 5 & 6 & 7 \\
 			3  &   & 0 & 2 & 13 & 42 & 101  & 205 & 372    \\
 			4   & & & 0 & 6 & 192 & 1858 & 10644&  44595 \\
 			5   & & & & 0 & 24 & 5388 & 204117 & 3458276   \\
 			6   & & & & &  0& 120 & 255180 & 46545915 \\
 			7   & & & & &  & 0 & 720 & 18699210\\
			\end{array}}
		\]
		\end{small}
		\caption{The number of bounded regions of $\mathcal B(k,m)$ for various $k$ and $m$.
		See Remark \ref{rmk:EulerBkm}.
	\label{tab:parallel}}
\end{table}

 It is the second arrangement, $\mathcal B(k,m)$, which plays the center stage for our application.
 We shall reduce its analysis to that of $\mathcal{A}(k,m)$, so we can use results 
 of Koizumi et al.~\cite{Koizumi}.

\begin{remark} \label{rmk:EulerBkm} The generic fiber of the map $\pi_{k,m}$ is the complement of
a hyperplane arrangement in complex projective space $\PP^{k-1}$ that is
combinatorially isomorphic to $\mathcal{B}(k,m)$. Hence the Euler characteristic of the
generic fiber equals the number of bounded regions of $\mathcal{B}(k,m)$.
\end{remark}

We begin with some basics from \cite{Stanley}.
Given an arrangement $\mathcal A$ of hyperplanes $H_1,\ldots, H_N$  in real affine space $\mathbb{R}^d$,
 we write $F_I = \bigcap_{i \in I} H_i$ for any subset $I $ of $ [N] := \{1,\ldots,N\}$. The collection of nonempty 
 $F_I$ forms a poset $L(\mathcal A)$ by reverse inclusion.
 This is called the {\em intersection poset}, and it is graded by the codimension of $F_I$. 
The {\em characteristic polynomial}~is
\begin{equation}\label{eq:char_poly}
\chi_{\mathcal A}(t)\,\,\, = \,\,\sum_{F_I \in L(\mathcal A)} \mu(F_I) \cdot t^{d-\textrm{codim}(F_I)} 
\,\,\,=: \,\,\,\sum_{i=0}^d (-1)^d\cdot b_i(\mathcal A)\cdot t^{d-i}
\end{equation}
where $\mu(F_I) = \mu(\RR^d,F_I)$ is the M\"obius function of $L(\mathcal{A})$.
 The coefficients $\{b_i(\mathcal A)\}_{i=0}^d$ are the {\em Betti numbers} of $\mathcal{A}$. They are
 nonnegative integers. The components of the complement of the arrangement in $\mathbb{R}^d$ are called the \emph{regions} of $\mathcal A$.
 Zaslavsky \cite{Zas75} showed that the regions are counted, up to sign,  by $\chi_{\mathcal A}(-1)$.
  The bounded regions of $\mathcal{A}$ are counted by $|\chi_{\mathcal A}(1) |$.

We are interested in the characteristic polynomials of the generic discriminantal arrangements $\mathcal{A}(k,m)$ and
$\mathcal{B}(k,m)$ defined above. The following is our main result in this section.

\begin{theorem} \label{thm:arepolynomials}
For fixed $k$, the Betti numbers of $\mathcal A(k,m)$ and $\mathcal B(k,m)$ are polynomials in
the parameter $m$.
Here the $i$-th Betti number is a polynomial of degree $i(k-1)$. The number of bounded regions
of either arrangement is given by a polynomial in $m$ of degree $(k-1)^2$.
\end{theorem}

In light of Remark \ref{rmk:EulerBkm}, we are primarily interested in the number of bounded regions of $\mathcal{B}(k,m)$.
Cachazo et al.~\cite[\S 3]{CUZ} wrote these as polynomials of
degree $4$ and $9$ for $k=3,4$. This led to the polynomiality conjecture that is proved here.
We list our polynomials for~$k \leq 7$:

\medskip

\noindent $k=3$:{\small \qquad
$\frac{1}{(2!)^3}(m-4)(m^{3}-6m^{2}+11m-14)$ \\
}\vspace*{-0.1cm}

\noindent $k=4$:
{\small
$ \frac{1}{(3!)^4}(m-5)(m^{8}-13m^{7}-5m^{6}+1019m^{5}-7934m^{4}+29198m^{3}-57510m^{2}+57276m-20736)$}
\\

\noindent $k=5$:
{\tiny
$\quad \frac{1}{(4!)^5}(m-6)(m^{15}-34m^{14}+536m^{13}-6016m^{12}+56342m^{11}-324124m^{10}-737436m^{9}+33755560m^{8}-
      324772079m^{7} \\ \phantom{dada} \qquad \qquad +1784683822m^{6}-6330080036m^{5}
      +14844484456m^{4}-22600207744m^{3}+21093515136m^{2}-
      10696725504m+2188394496)$
}

\noindent $k=6$:
{\tiny
$\quad \frac{1}{(5!)^6}(m-7)(m^{24}-68m^{23}+2199m^{22}-44982m^{21}+643996m^{20}-6596728m^{19}+44404954m^{18}-71800572m^{
      17}-3127304119m^{16} \\ \phantom{dada}  \qquad \qquad
      +54378585092m^{15}
      -490514132181m^{14}+1776590470858m^{13}+18487658083746m^{12}
      -378944155004728m^{11}+3596483286643204m^{10} \vspace{-0.12cm}  \\ \phantom{dada}  \qquad \qquad
       -23256569552060072m^{9}
      +111789542126956376m^{8}-
      412236512413133568m^{7}+1177950138000941824m^{6}-2598297935794415232m^{5} \vspace{-0.12cm}  \\ \phantom{dada} \qquad \qquad
            +4348400923758960000m^{4}      -5332505009742720000m^{3}+4499090747884800000m^{2}-2314865713121280000m+540696010752000000)$
      }
      
\noindent $k=7$:
{\tiny
$ \quad \frac{1}{(6!)^7}(m-8)(m^{35}-118m^{34}+6721m^{33}-246112m^{32}+6510988m^{31}-132696712m^{30}+2174185716m^{29}-
      29636017152m^{28}
\\ \phantom{dada} \quad      
      +347341371054m^{27}-3619686930036m^{26}
      $$+34412884983870m^{25}-297168540668160m^{
      24}+2163013989971700m^{23}-9541316565707160m^{22} \vspace{-0.12cm} \\ 
      -57505724652057900m^{21}+2082608669563706400m^{
      20} -30141885880045284135m^{19}+290868328355522261370m^{18}-1789045412885731209655m^{17}
    \vspace{-0.12cm}   \\ \phantom{dada} \quad +
      2410419545362804828960m^{16}+98045335860104486976824m^{15}
      -1463950567334046862107632m^{14}+
      13043235985765459517500784m^{13} \vspace{-0.12cm} \\ 
      \phantom{dada} \quad
      -86372694744783734157998048m^{12}+451792193002382546052610752
      m^{11} -1910783375678688560288108928n^{10} \vspace{-0.12cm} \\ \phantom{dada} \quad
      +6590586320126085204961058304n^{9} -
      18548788372955608600206309888m^{8}+42375051122462107996615842816m^{7} \vspace{-0.12cm} \\ \phantom{dada} \quad
     -   77750332979800481270501750784m^{6}+112672771830431700211895132160m^{5}-
      125767704870432247378231296000m^{4} \vspace{-0.1cm} \\ \phantom{dada} \quad +104081289628186508359680000000m^{3} -
      60012326911967527500840960000m^{2}+21475615998699753858662400000m- 2^25 3^15 5^7 23 593 6983) $}
      
\medskip

\noindent Each polynomial gives the values in one row of Table \ref{tab:parallel}.
See also the list at the end of~\cite{Koizumi}.

We now embark towards the proof of Theorem \ref{thm:arepolynomials}, beginning with the following lemma.

\begin{lemma}
\label{lem:characteristicpolyfromAtoB}
The characteristic polynomials of $\mathcal A(k,m)$ and $\mathcal B(k,m)$ are related as follows:
$$\chi_{\mathcal A(k,m)}(t)\cdot t \,-\, \chi_{\mathcal A(k,m)}(1) \,\,\,=\,\,\, \chi_{\mathcal B(k,m)}(t)\cdot(t-1).$$ 
\end{lemma}

\begin{proof}
The central arrangement $\widetilde{\mathcal B}(k,m)$ is the cone over $\mathcal B(k,m)$.
This gives the relation
$$\chi_{\widetilde{\mathcal B}(k,m)}(t) \,\,=\,\, \chi_{\mathcal B(k,m)}(t)\cdot(t-1).$$ 
The restriction $\mathcal{C}^H$
of any central arrangement $\mathcal C$ to a hyperplane $H$ not containing the origin 
satisfies $\,\chi_{\mathcal C^H}(t)\cdot t \,-\, \chi_{\mathcal C^H}(1) = \chi_{\mathcal C}(t)$.
We apply this fact to $\,\mathcal{C} = \widetilde{\mathcal{B}}(k,m)$ and
$H = \{x_1 = 1\}$. This implies the assertion because the restriction of $\mathcal{C}$ to $H$ is precisely
$\,\mathcal{C}^H = \mathcal A(k,m)$.
\end{proof}

\begin{example}\label{example:A34} 
The discriminantal arrangement $\widetilde{\mathcal{B}}(3,4)$ given by four vectors in $\RR^3$ satisfies
$$ \chi_{\widetilde{\mathcal B}(3,4)}(t) \,\,=\,\, t^3-6t^2+11t-6.$$
The characteristic polynomials  in Example \ref{ex:A34greenblue} are derived from this.
We obtain $\chi_{\mathcal A(3,4)}(t) $ by deleting the constant term and dividing by $t$.
We obtain $\chi_{\mathcal B(3,4)}(t) $ by dividing by $t-1$.
\end{example}

\begin{corollary}
\label{cor:boundedBkn}
The number of bounded regions of $\,\mathcal B(k,m)\,$ is given, up to a sign, by
$$ \chi_{\mathcal B(k,m)}(1) \,\,=\,\, \frac{d}{dt}\chi_{\mathcal A(k,m)}(1) \,+\, \chi_{\mathcal A(k,m)}(1).$$
\end{corollary}

Koizumi et al.~\cite{Koizumi} provide  formulas for $\chi_{\mathcal A(k,m)}$ for $1 \leq k \leq 7$ and general $m$.
The coefficients of their characteristic polynomials, i.e.~the Betti numbers,
 are  polynomials in $m$. We show that this trend continues. The extension to $\mathcal B(k,m)$ 
 follows from Lemma \ref{lem:characteristicpolyfromAtoB}.

\begin{lemma}
The Betti numbers of both $\,\mathcal A(k,m)\,$ and $\,\mathcal B(k,m)\,$ are polynomials in $m$.
\end{lemma}

\begin{proof}
We shall use \cite{Koizumi} to derive the result for $\mathcal A(k,m)$. We set $d=k-1$
and write $L(d,m)$ for the intersection poset  of $\mathcal A(k,m)$, with M\"obius function $\mu_{d,m}$.
Note that $\mu_{d,m}(\{\emptyset\})$ is the number of bounded regions in $\mathcal A(k,m)$.
In \cite{Koizumi}, the elements $F$ of $L(d,m)$ are grouped according to partitions $\gamma=(\gamma_1,\ldots,\gamma_\ell)$, called the \emph{type} of $F$. The number of elements of type $\gamma$ is denoted $\lambda_{d,m}(\gamma)$.
  A main result of \cite{Koizumi} is that the M\"obius function applied to $F$ depends only on the type of $F$. Hence, the characteristic polynomial of $L(d,m)$ equals
\begin{equation}
\label{eq:KoizumiCharSimple}
\chi_{d,m}(t) \,\,= \,\,
 \sum_{i=0}^d \sum_{\gamma \vdash i} \lambda_{d,m}(\gamma) \mu_{d,m}(\gamma) t^{d-i}.
\end{equation}

We will show that $\lambda_{d,m}(\gamma)$ and $\mu_{d,m}(\gamma)$ are both polynomials in $m$ for fixed $d$.
Then \eqref{eq:KoizumiCharSimple} implies that the coefficients of the characteristic polynomial are polynomials in $m$ as well. 
Proposition 4.7 of \cite{Koizumi} gives the following formula:
	\begin{equation}\label{eq:KoizumiLambda}
\lambda_{d,m}(\gamma)\, \,\,=\,\, \, \frac{1}{\prod_{k=1}^d m_k(\gamma)!}\, 
\sum_{\nu} \frac{(\nu(I_2)+\cdots+\nu(I_{2^\ell}))!}{\nu(I_2)!\cdots \nu(I_{2^\ell})!}\,{{m}\choose{\nu(I_2)+\cdots+\nu(I_{2^\ell})}}. 
	\end{equation}
Here, $m_k(\gamma)$ is a constant depending on $k$ and $\gamma$, and $2^{[\ell]}$ is the power set 
of $ [\ell]$.
  This set is ordered $\{I_1,\ldots,I_{2^{\ell}}\}$ with $I_1 = \emptyset$. The sum in \eqref{eq:KoizumiLambda} is over all maps $\nu:2^{[\ell]} \to \mathbb{N}$~satisfying
\begin{enumerate} 
\item $\sum_{i \in I} \nu(I)\, =\, d+1 -\gamma_i\,$ for all $\,i \in [\ell]$,
\item $\sum_{I \subset I'} \nu(I') \,<\, d+ 1 - \sum_{i \in I} \gamma_i\,$ for all $I$ with $|I|>1$, and 
\item $\sum_{I \in 2^{[\ell]}} \nu(I) \,=\, m$.
\end{enumerate} Note that as $m$ grows by one, the indexing set of these maps essentially remains the same; the only difference is that $\nu(\emptyset)$ is also incremented by one. Since $\nu(\emptyset)=\nu(I_1)$ does not appear in the binomial in \eqref{eq:KoizumiLambda}, the expression \eqref{eq:KoizumiLambda} is polynomial in $m$.

 Proposition 4.1 of \cite{Koizumi} gives a recursive definition for $\mu_{d,m}(\gamma)$ in terms of $\mu_{d',m'}(\{\emptyset\})$ for strictly smaller $d'$ and $m'$. Since these are the numbers of bounded regions of $\mathcal A(k',m')$ for strictly smaller $k'$ and $m'$, their polynomiality follows inductively, proving the result.
 \end{proof}
 
 \begin{proof}[Proof of Theorem \ref{thm:arepolynomials}]
 The Betti number $b_i$ of interest is the coefficient of $t^{d-i}$ in $\chi_{d,m}(t)$.
We know that this is a polynomial in $m$. We first argue that the degree of this polynomial is
 at most $d\cdot i$, and then we show that it has degree at least $d \cdot i$. 
For the upper bound, we note that 
the $i$-th Betti number of a generic arrangement with ${{m}\choose{d}}$ hyperplanes in $\RR^d$
is a polynomial in $m$ of degree $d \cdot i$.
 Any other arrangement with the same number of hyperplanes,
 including discriminantal ones, must have Betti numbers bounded by these generic Betti numbers. 

To derive the lower bound, we recall that, by definition of the Betti numbers in~\eqref{eq:char_poly},
\begin{equation}
b_i \,\,\,= \sum_{F \in L(d,m) \atop \text{codim}(F)=i} \!\! \mu_{d,m}(F).
\end{equation}
All terms in this sum have the same sign. Suppose $m>d\cdot i$ and consider all 
$T_{e_1},\ldots,T_{e_i} \subset [m]$ where each $T_{e_j}$ consists of $d$ points and $T_{e_j} \cap T_{e_k} = \emptyset$ for $1 \leq j < k \leq i$. These collections correspond to the codimension $i$ flats $\bigcap_{j=1}^i H_j$ where $H_j$ is the hyperplane spanned by the points indexed by $T_{e_j}$.
These are special flats in our arrangement.
 Their number is 
 $$\frac{1}{d!}\underbrace{{{m}\choose d}\cdot {{m-d}\choose{d}} \cdot \cdots  \cdot {{m-id+d}\choose{d}}}_{i \text{ factors}}.$$
 This product of $i$ binomial coefficients is a polynomial in $m$ of degree $d \cdot i$.
Therefore, the degree of $b_i $ as polynomial in $m$
 is at least $d\cdot i$. This completes the proof of Theorem \ref{thm:arepolynomials}.
\end{proof}

\begin{corollary}
The ML degree of  the generic fiber of $\pi_{k,m}$
 is a polynomial of degree $(k{-}1)^2$.
\end{corollary}

We presented formulas for these ML degrees for $k \leq 7$ and arbitrary $m$.
These play a major role in the stratified fibration approach of Section \ref{section2}. 
However, in addition to the generic fiber, we also need to know the Euler
characteristic for the fiber over each stratum in the base
space $X(k,m)$. All of these fibers are discriminantal arrangements, 
arising from lower-dimensional matroid strata in the large Grassmannian on the right hand side of 
(\ref{eq:mapgrass}). In other words, we need to compute $\chi_{\cal A}(1)$ for 
many large hyperplane arrangements ${\cal A}$.

In practise, this task can now be accomplished easily, thanks to the software 
recently presented in \cite{BHK}. This implementation was essential to us
in getting this project started, and in
 validating the polynomial formulas for $\mathcal{B}(k,m)$ displayed above.
 We expect it to be useful for a wide range of applications, not just in mathematics, but
 also in physics and statistics.

\section{Matroid strata} \label{section4}

In our topological approach to computing the ML degree of
$X(k,m)$, we encountered special strata of
points and lines in $\PP^{k-1}$ with prescribed incidence conditions.
Such strata can be modeled as matroid strata. From now on, all matroids
are assumed to be simple, so the term ``matroid'' will mean ``simple matroid''.
We shall  comprehensively study all small~matroids. 

For a matroid $M$ of rank $k$ on $m$ elements, we consider the 
 points in ${\rm Gr}(k,m)$ whose nonzero Pl\"ucker coordinates
are precisely those indexed by the bases of~$M$.
Let $X(M)$ be the quotient of that constructible set modulo
the action of $(\CC^*)^m$. One represents
$X(M)$ by a $k \times m$ matrix with at least one
entry~$1$ per column and some unknown 
entries that satisfy equations and inequations of degree $\leq k-1$
arising from nonbases and bases.
This encoding shows that $X(M)$ is a very affine variety.
If $M$ is the uniform matroid then $X(M) = X(k,m)$.
The aim of this section is to compute the ML degree of $X(M)$
for every matroid of small~size.

\begin{example}[$k=3,m=9$] \label{ex:pappus}
The {\em Pappus matroid} $M$ (shown in Figure \ref{fig:pappus})
has the nonbases
$$ 123 \quad 148\quad 159\quad 247\quad 269\quad 357\quad 368\quad 456\quad 789.
$$
These are precisely the triples that index the vanishing $3 \times 3$ minors of the matrix
\begin{equation}
\label{eq:pappusmatrix}
\begin{bmatrix}
 \,\, 1 & 0 & 1 & 0 & 1 & y &   0     &    1       & y  \,\,\\
\,\, 0 & 1 & x & 0 & 1 & y &   1     &    0       & 1  \,\, \\
\,\, 0 & 0 & 0 & 1 & 1 & 1 & 1/(1{-}x) & x/(y(x{-}1)) & 1 \,\, 
\end{bmatrix}.
\end{equation}
Indeed, the stratum $X(M)$ for the Pappus matroid is the set of
pairs $(x,y) \in (\CC^*)^2$ such that all $\,\binom{9}{3} - 9 = 75\,$ other maximal minors are nonzero.
The log-likelihood function equals
$$ 
u_1 {\rm log}(x) + u_2 {\rm log}(y)  + u_3 {\rm log}(1-x) + u_4 {\rm log}(1-y) + u_5 {\rm log}(1-x-y) + u_6 {\rm log}(1-xy) + 
u_7 {\rm log}(xy-x-y).
$$
The two partial derivatives of this expression are rational functions in $x$ and $y$. By equating these to zero,
we obtain a system of equations that has precisely eight solutions in $(\CC^*)^2$, provided
$u_1,u_2,\ldots,u_7$ are general enough. Hence the
stratum of the Pappus matroid $M$ is a very affine variety
$X(M)$ of dimension $2$ whose ML degree is equal to $8$.
It is one of the $12$ matroids of rank $3$ on $9$ elements with these invariants,
marked $8_{12}$ in Theorem~\ref{thm:matroid_ml}.
\end{example}

We now present the result of our computations for 
matroids of rank $k$ on $m$ elements. 
For fixed $k,m$, we list the matroid strata by dimension,
and we list  all occurring ML degrees
together with their multiplicity of occurrences.
For instance, the string $ 2_5  [1 ,2_4]$ accounts for
five strata of dimension $2$: four have ML degree $2$
and one has ML degree $1$. If the variety $X(M)$ is reducible, then we list
the ML degree for each component, e.g.~in the format $(3,3)_2$.

Below we follow the convention in the  applied algebraic geometry literature (e.g.~\cite{OS})
to assign a star to a theorem obtained by a numerical computation
which was not fully certified.

\begin{theoremstar}\label{thm:matroid_ml}
The strata $X(M)$ are  smooth for all matroids $M$ with
$k = 2$ or ($k=3$ and $m \leq 9$) or  ($k=4$ and $m=8$). 
Their ML degrees are given in the following lists.

\smallskip
\noindent
For $k=3,m=5$ there are $4$  matroids, up to 
permuting labels, and all are realizable over~$\CC$:
\vspace{-0.1in} 
$$ \begin{matrix} 2_1  [2_1] &&  1_2  [1_2] & &   0_1  [1_1] \end{matrix}  . $$
  For $k=3,m=6$ there are $9$  matroids, up to 
permuting labels, and all are realizable over~$\CC$:
\vspace{-0.1in} 
$$ \begin{matrix}
4_1  [26_1] &&  3_1  [6_1] & &   2_4  [1_1 ,2_3] & & 1_2 [1_2] &&  0_1  [1_1]  
\end{matrix}  . $$
For $k=3,m=7$ there are $22$ orbits of $\CC$-realizable  matroids:
$$ \begin{matrix}
6_1   [1272_1] &
5_1   [192_1] &
4_3    [24_2, 38_1] & 
3_7   [2_1, 6_4 ,  10_1, 12_1]  &
2_{6}  [ 1_1, 2_3, 4_2 ] &
1_3  [1_2, 2_1 ] &
0_1 [1_1]  \end{matrix} 
$$
For $k=3,m=8$ there are $66$ orbits of $\CC$-realizable  matroids:
$$ \begin{matrix} \begin{matrix}
 8_1 [188112_1] &
7_1 [21240_1] &
6_3 [1560_1, 2136_1, 2976_1] &
5_7 [120_2, 264_2, 368_1, 520_1, 568_1] 
\end{matrix}  \\
 \begin{matrix}
4_{16} [4_1, 6_1, 24_5,  38_1,  55_1, 56_2, 72_1, 80_1, 88_1, 120_2 ]  & 
3_{16} [2_1, 6_{4}, 10_2, 16_3, 17_1, 18_2, 24_1, 25_1, 32_1 ] 
\end{matrix}
\\
 \begin{matrix} \qquad
2_{14} [2_{4},4_3 ,  5_1, 6_3, 7_1, 8_2] & \qquad
1_{6} [1_{1},2_4, 3_1 ] & \qquad
0_{2} [1_1,2_1 ] 
\end{matrix}
\end{matrix}
$$
For $k=3,m=9$ there are $368$
orbits of $\CC$-realizable matroids:
{\footnotesize
$$\begin{matrix}
10_1 [\underline{74570400_1}] \quad  9_1[\underline{6750000_1}] \quad 8_{3}[349920_{1},\underline{565706_{1}},\underline{730656_{1}}], \\
7_{8}[13920_{1},24624_{1},43512_{1},44496_{1},56448_{1},73024_{1},94668_{1},99808_{1}]
\\
6_{22}[720_{3},2040_{2},2856_{2},3696_{1},5160_{1},6648_{1},6700_{1},6708_{1},6796_{1}, 7752_{1},8646_{1},\\ 8684_{1},9152_{1},11392_{1},11972_{1},14880_{1}, 15132_{1},18768_{1}]\\
5_{45}[0_{2},12_{1},24_{1},120_{6},264_{2},360_{2},480_{1},500_{1},512_{2},672_{1},712_{2},720_{1},780_{1},922_{1},948_{1}, 956_{1},960_{1},1092_{1},\\ 1220_{1},1264_{1},1294_{1},1296_{3},1316_{1},1364_{1},1672_{1},1736_{1},1744_{1},
1756_{1},1806_{1},2274_{1},2292_{1},2628_{1},2648_{1}]\\
4_{77}[0_{1},4_{1},6_{1},24_{5},38_{1},56_{3},80_{4},82_{2},96_{2},116_{2},118_{2},119_{1},120_{1},124_{3},132_{1},144_{1},150_{1},156_{1},162_{2},166_{1},\\
168_{1},172_{2},174_{1},186_{1},192_{1},200_{1},222_{1},224_{1},228_{2},232_{3},236_{1},238_{1},
240_{1},242_{2},243_{1},244_{1},262_{1}, \\ 272_{1}, 280_{1},
292_{1},302_{1},306_{1},312_{1},318_{1},324_{1},330_{2},336_{1},342_{1},
376_{1},416_{1},424_{2},434_{1},456_{1},460_{1}]\\
3_{93}[0_{2},6_{4},10_{3},16_{6},17_{1},18_{2},24_{6},25_{1},26_{4},27_{1},30_{1},32_{1},34_{2},35_{2},
36_{5},37_{3},38_{3},39_{1},40_{5},42_{2},44_{2},46_{2},\\
47_{2},51_{1},52_{4},54_{2},55_{1},56_{4},58_{2},60_{1},
64_{1},66_{2},68_{2},70_{3},72_{2},76_{2},84_{1},86_{1},102_{1},104_{1},108_{1}] \\
2_{78}[0_{1},2_{3},4_{4},6_{8},7_{3},8_{12},9_{3},10_{10},11_{4},12_{7},
13_{3},14_{2},15_{3},16_{7},17_{1},18_{3},19_{2},30_{1},(18, 18)_{1}],\\
1_{34}[0_{3},2_{10},3_{8},4_{7},5_{2},6_{2},(3, 3)_{2}]
\qquad \qquad 0_{8}[1_{3},2_{5}]
\end{matrix} $$}
For $k=4,m=8$ there are $554$
orbits of $\CC$-realizable matroids:
{\footnotesize
$$\begin{matrix}
9_{1}[\underline{5211816}] \qquad
8_{1}[\underline{516673_{1}}] \qquad
7_{5}[21240_{2},46392_{1},52392_{1},63040_{1}]\\
6_{14}[1272_{2},1560_{2},2136_{1},2976_{2},5136_{1},6976_{2},7600_{1},9424_{1},10368_{2}]\\
5_{40}[26_{1},120_{4},192_{4},264_{6},368_{3},488_{1},520_{2},568_{2},692_{1},708_{1},770_{1},936_{1},950_{1},\\
1080_{3},1294_{1},1296_{2},1322_{1},1400_{1},1710_{1},1768_{1},1812_{2}]\\
4_{89}[0_{3},4_{1},6_{2},24_{14},26_{1},38_{9},55_{1},56_{8},72_{3},80_{3},88_{4},104_{1},115_{2},120_{6},122_{1},
134_{2},136_{1},154_{3},\\162_{2},180_{2},181_{1},216_{1},220_{1},224_{1},226_{1},232_{2},236_{2},272_{1},282_{3},288_{2},
308_{2},384_{2},410_{1}]\\
3_{153}[1_{1},2_{8},6_{30},10_{15},12_{5},16_{13},17_{5},18_{10},24_{9},25_{5},27_{2},28_{1},32_{7},34_{1},35_{3},
36_{5},38_{1},\\
39_{2},40_{1},45_{1},46_{2},48_{1},50_{6},51_{1},52_{3},53_{2},55_{1},56_{2},58_{2},64_{2},66_{2},67_{1},72_{2},104_{1}]\\
2_{153}[1_{8},2_{35},4_{29},5_{1},6_{18},7_{10},8_{16},9_{2},10_{6},11_{6},12_{8},13_{3},14_{3},15_{3},18_{2},19_{1},24_{2}]\\
1_{81}[0_{4},1_{20},2_{36},3_{9},4_{9},5_{2},(2, 2)_{1}] \qquad \qquad
0_{18}[1_{15},2_{3}]
\end{matrix}
$$}
\end{theoremstar}

\begin{proof}
This was obtained by exhaustive computations.
The matroids were taken from the database~\cite{MMIB_DB} that is
described in \cite{MMIB}.
The
 \texttt{GAP} packages \texttt{alcove}~\cite{alcove} and \texttt{ZariskiFrames}~\cite{ZariskiFrames}
 were used to obtain a representing $k \times m$ matrix, such as (\ref{eq:pappusmatrix}),
  along with further equations for nonbases and inequations for bases, whenever needed.
  The details of the underlying algorithm are described in~\cite{BK}.
From this matrix, together with the defining equations, we compute the ML degree as the number of critical points of the
log-likelihood function. This is illustrated in  Example \ref{ex:pappus}. These computations were performed 
using the \texttt{Julia} package {\tt HomotopyContinuation.jl}~\cite{HCjl}.
 See also \cite{ST}.
The ML degrees of the uniform matroids $U_{3,7}$, $U_{3,8}$, $U_{3,9}$, $U_{4,8}$ are derived  and discussed in Sections~\ref{section5} and~\ref{section6}.
\end{proof}

\begin{remark}
As the method of homotopy continuation is numerical, it is inherently subject to 
rounding errors. Hence, the numbers in Theorem \ref{thm:matroid_ml} come with this disclaimer as well. For those numbers which are not underlined, we successfully performed the certification method described in \cite{BRT}, using the implementation conveniently available in~\cite{HCjl}. Such a certified result delivers a mathematical \emph{proof} that our number bounds the true ML degree from below. Of the six underlined entries, two are the uniform matroids $U_{3,9}$ and $U_{4,8}$ (discussed in Sections \ref{section5} and \ref{section6}) and another is a matroid whose ML degree is computed in Example \ref{ex:NonUniformMLDegree}. The certificates proving the correctness of our results can be found at \url{https://zenodo.org/record/7454826#.Y6DCdezMKdY}
\end{remark}

We now briefly discuss some special matroids that appear in our lists,
shown in Figure~\ref{fig:matroids}.

\begin{example}
\label{ex:specialmatroids}
The {\em Pappus matroid} was seen in Example \ref{ex:pappus}.
The {\em non-Fano matroid}  has $k=3$ and $m=7$. It is projectively unique, so
its stratum has dimension $0$ and ML degree $1$.
The {\em affine geometry  $AG(2,3)$} has $k=3$ and $m=9$.
	 It is also known as the dual Hessian configuration.
	Its stratum is $0$-dimensional of degree $2$. Hence, its ML degree is $2$.
The table for $k=3, m=9$ lists nine
$\mathbb{C}$-realizable  matroids with ML degree $0$.
	The one with the fewest nonbases has $6$ nonbases; its stratum is $4$-dimensional and has degree $4$.
	A geometric representation of this matroid has the form of $3\times 3$ grid and is depicted in Figure~\ref{fig:grid}.
\end{example}

\begin{figure}[h]
		\begin{subfigure}{0.24\linewidth}
		\centering
		\includegraphics[scale=.65]{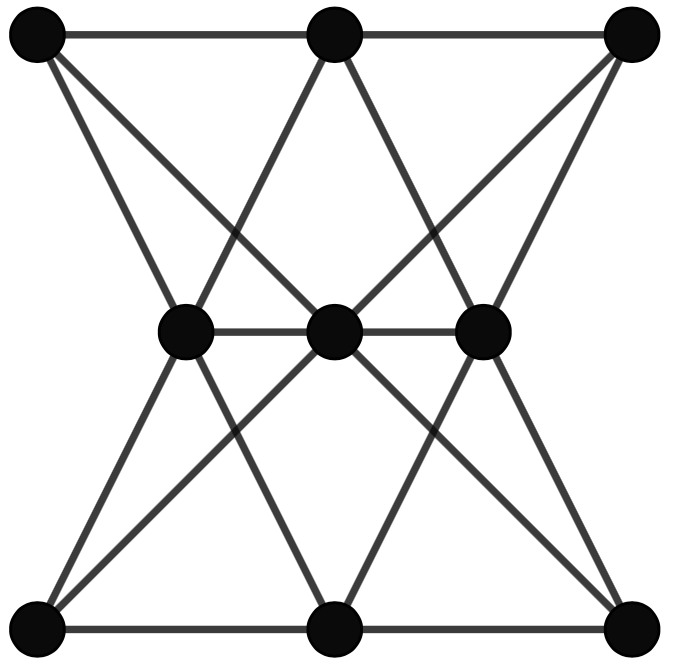}
		\caption{Pappus matroid}
		\label{fig:pappus}
	\end{subfigure}
\begin{subfigure}{0.24\linewidth}
		\centering
		\includegraphics[scale=.7]{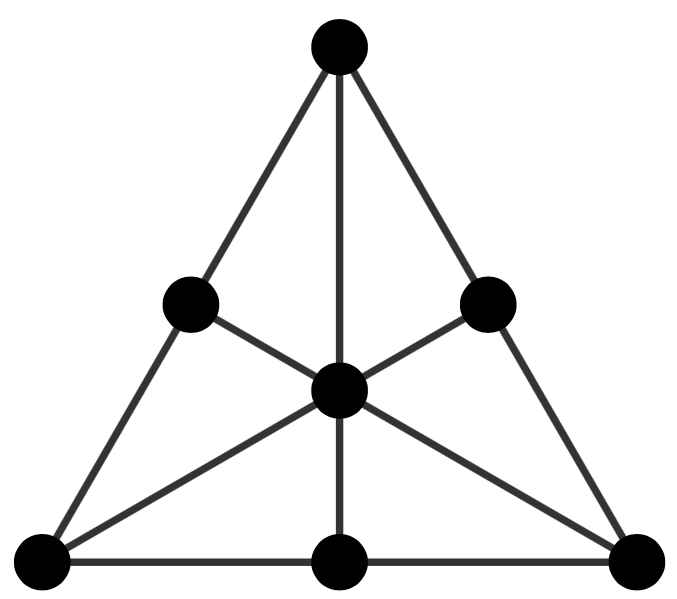}
		\caption{Non Fano matroid}
	\end{subfigure}
		\begin{subfigure}{0.24\linewidth}
	\centering
	\includegraphics[scale=.2]{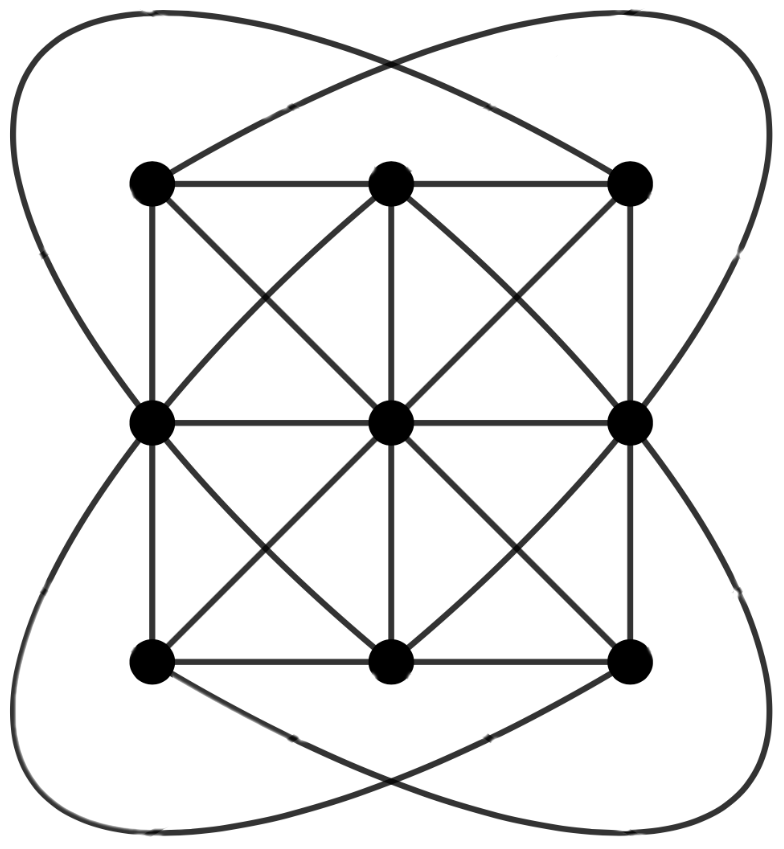}
	\caption{$AG(2,3)$}
	\end{subfigure}
	\begin{subfigure}{0.26\linewidth}
	\centering
	\includegraphics[scale=.65]{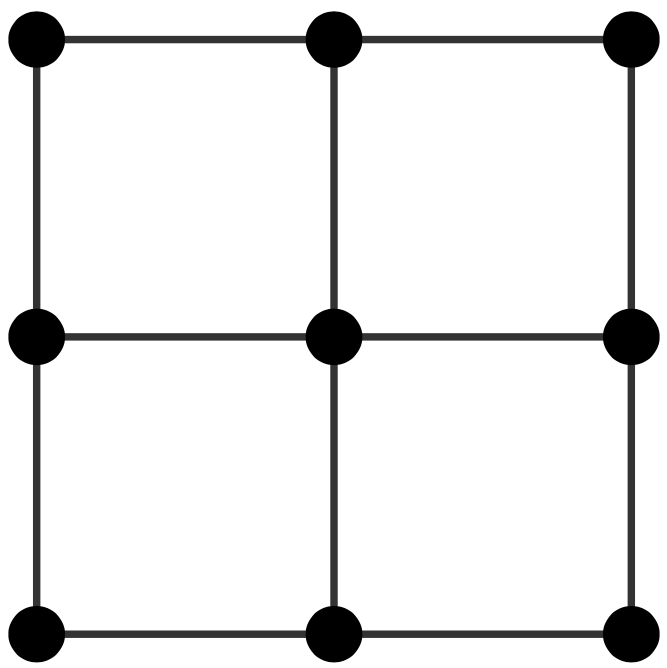}
	\caption{ML degree $0$ matroid}
	\label{fig:grid}
	\end{subfigure}
	\caption{Geometric representations of the matroids discussed in Examples
	\ref{ex:pappus} and \ref{ex:specialmatroids}.}
	\label{fig:matroids}
\end{figure}

\section{Fibrations of points and lines in the plane} \label{section5}

We now study the moduli space $X(3,m)$ of
$m$ labeled points in $\PP^2$ in linearly general position.
This is  very affine of dimension $2m-8$. Here is what we know about its
Euler characteristic.

\begin{theorem} \label{thm:numbersk=3} The ML degree of $X(3,m)$ is given by the following table for $m \leq 9$:
	$$
	\begin{matrix}
	m                &\vrule&  4 \,\, & 5 \,\,& 6 \,\,& 7 & 8 & 9 \\ \hline
	\chi(X(3,m)) &\vrule & 1 \,\, & 2 \,\,& 26 \,\,& 1\,272 & 188\,112 & 74\,570\,400 \\
	\end{matrix}	
	$$
\end{theorem}

These numbers are easy to prove for $m \leq 6$. A computational proof for $m=6$
appeared in \cite[Appendix C]{CEGM}.
The numbers for $m=7,8$ were derived with the soft
limit argument in  \cite{CUZ}. This was a proof in the
sense of physics but perhaps not in the sense of mathematics.
The verification by numerical computation 
was presented in \cite[Proposition 5]{ST}. Thomas Lam had derived and proved
all numbers, including the $m=9$ case, using finite field methods.
This work was mentioned in the introduction of \cite{CUZ}, but 
Lam's finite fields proof had been unpublished so far.
It now appears along with this article, namely in Appendix \ref{appendixA} below.

The aim of this section is to solve this problem
   for general $m$ using the
stratified fibration approach in Section~\ref{section2}. 
Our techniques will be of independent interest.
In particular, they can be adapted to give a geometric proof for each of
the ML degrees of matroids in Theorem \ref{thm:matroid_ml}.
As a warm-up, here are geometric proofs for the first three numbers in
Theorem~\ref{thm:numbersk=3}.

\begin{example}[$m=4,5,6$] The space $X(3,4)$ is just a point, hence $\chi(X(3,4)) = 1$. 
	The very affine surface $X(3,5)$ is the complement of the 
	arrangement $\mathcal{B}(3,4)$ in Example \ref{ex:A34greenblue}.
	We have $\chi(X(3,5)) = 2$, by Varchenko's Theorem,
	as there are two bounded regions on the right in
	Figure \ref{fig:twoarrangements}. For $m=6$, we consider the deletion map
	$\pi_{3,6}: X(3,6) \rightarrow X(3,5)$. The discriminantal
	arrangement for any five points in $\PP^2$, with no three on a line,
	is isomorphic to $\mathcal{B}(3,5)$.
	All fibers of $\pi_{3,6}$ are homotopy equivalent to the
	complement of $\mathcal{B}(3,5)$.
	Hence $\pi_{3,6}$ is a fibration, where 
	each fiber has Euler characteristic $13$, by Table \ref{tab:parallel}.
	The base $X(3,5)$ has Euler characteristic $2$. Hence their product $26$
	is the Euler characteristic of $X(3,6)$.
\end{example}

We now consider $m \geq 6$.
Following Section \ref{section2}, we study the stratification $\mathcal{S}_{3,m}$ of $X(3,m)$.
The codimension one strata are all combinatorially equivalent: they are
the loci of configurations $[P]=[P_1,\dots,P_m]$ where three lines  $\overline{P_i P_j}$,  $\overline{P_k P_l}$, 
$\overline{P_r P_s}$ meet in a new point in $\PP^2$. We write  $S_{(ij)(kl)(rs)}$ for this stratum in $X(3,m)$.
All other strata are intersections of those, hence they can be denoted by a collection of triples $(ij)(kl)(rs)$.
For $m=6$, there are $15$ distinct codimension one strata $S_{(ij)(kl)(rs)}$,
one for each tripartition of $\{1,\dots,6\}$. All other strata in $\mathcal{S}_{3,6}$
are obtained by intersecting these $15$ divisors.
We found that $\mathcal{S}_{3,6}$ has
two combinatorially distinct codimension two strata, 
two codimension three strata and two codimension four strata. 
In Table \ref{table:strata} we list all strata explicitly, up to combinatorial equivalence.
Figure \ref{fig:configurations} shows point configurations $[P_1,\dots,P_6]$ 
for three among the seven strata in our list.

\begin{table}[h]
	\centering
	\begin{small}
		\begin{tabular}{c|c|c}
			Type &Codim & \multicolumn{1}{l}{Representatives for the divisors $(i,j)(k,l)(r,s)$ that intersect in the stratum}\\\hline
			I &1  & \multicolumn{1}{l}{(12)(34)(56)}\\
			II&2  & \multicolumn{1}{l}{(12)(34)(56),(12)(35)(46)}             \\
			III&2  & \multicolumn{1}{l}{(12)(34)(56),(15)(23)(46),(14)(26)(35)} \\
			IV&3  & \multicolumn{1}{l}{(12)(34)(56),(15)(23)(46),(14)(26)(35),(15)(26)(34)} \\
			V&3  & \multicolumn{1}{l}{(12)(34)(56),(12)(35)(46),(13)(26)(45),(14)(25)(36),(15)(24)(36),(16)(23)(45)}  \\
			VI&4  & \multicolumn{1}{l}{(12)(34)(56),(12)(35)(46),(14)(23)(56),(14)(26)(35),(15)(23)(46),(15)(26)(34)}\\
			VII&4  &   \multicolumn{1}{l}{(12)(34)(56),(12)(35)(46),(13)(24)(56),(13)(26)(45),(14)(25)(36),}\\ 
			&   & \multicolumn{1}{l}{(14)(26)(35),(15)(23)(46),(15)(24)(36),(16)(23)(45),(16)(25)(36)}                 
		\end{tabular}
	\end{small}
	\caption{All  types of strata in the $4$-dimensional very affine variety $X(3,6)$.
		\label{table:strata}}
\end{table}

\begin{figure}[h] $$ \hspace{-1.6cm}
	\includegraphics[scale=.35]{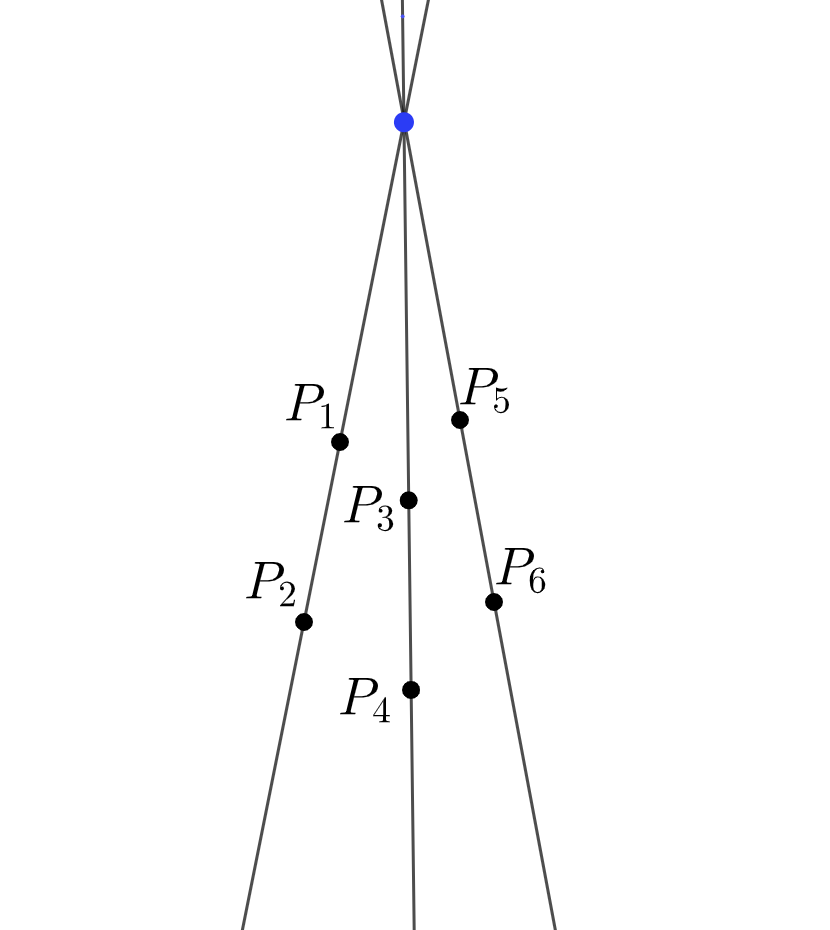} \!\!
	\includegraphics[scale=.33]{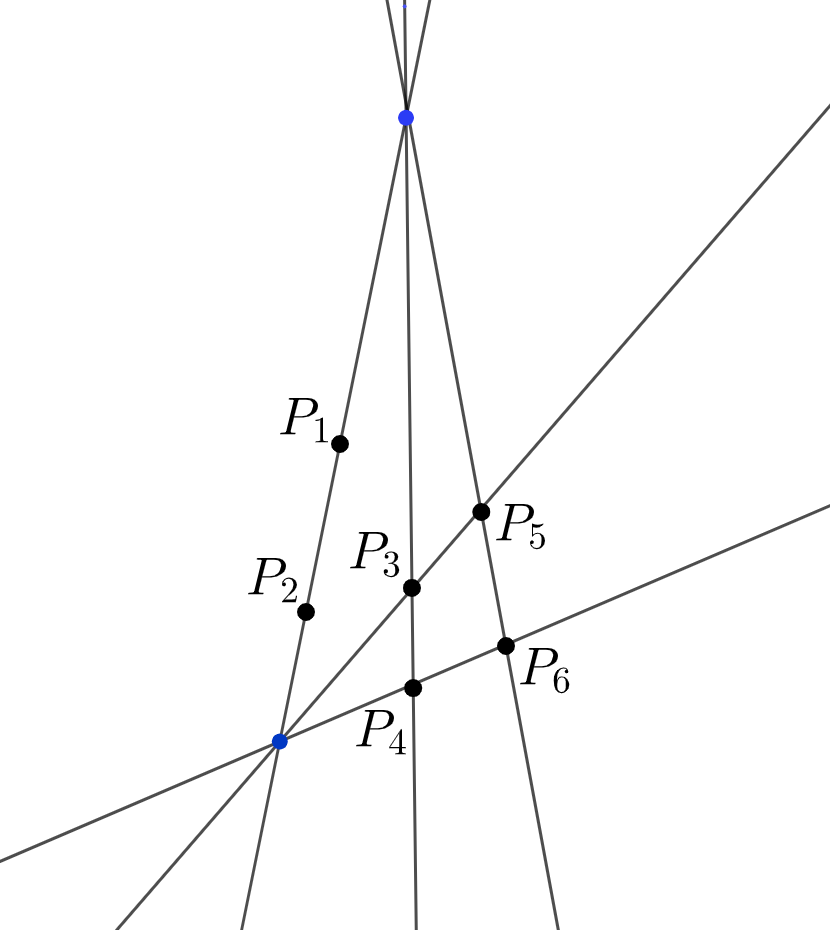} \qquad
	\includegraphics[scale=.286]{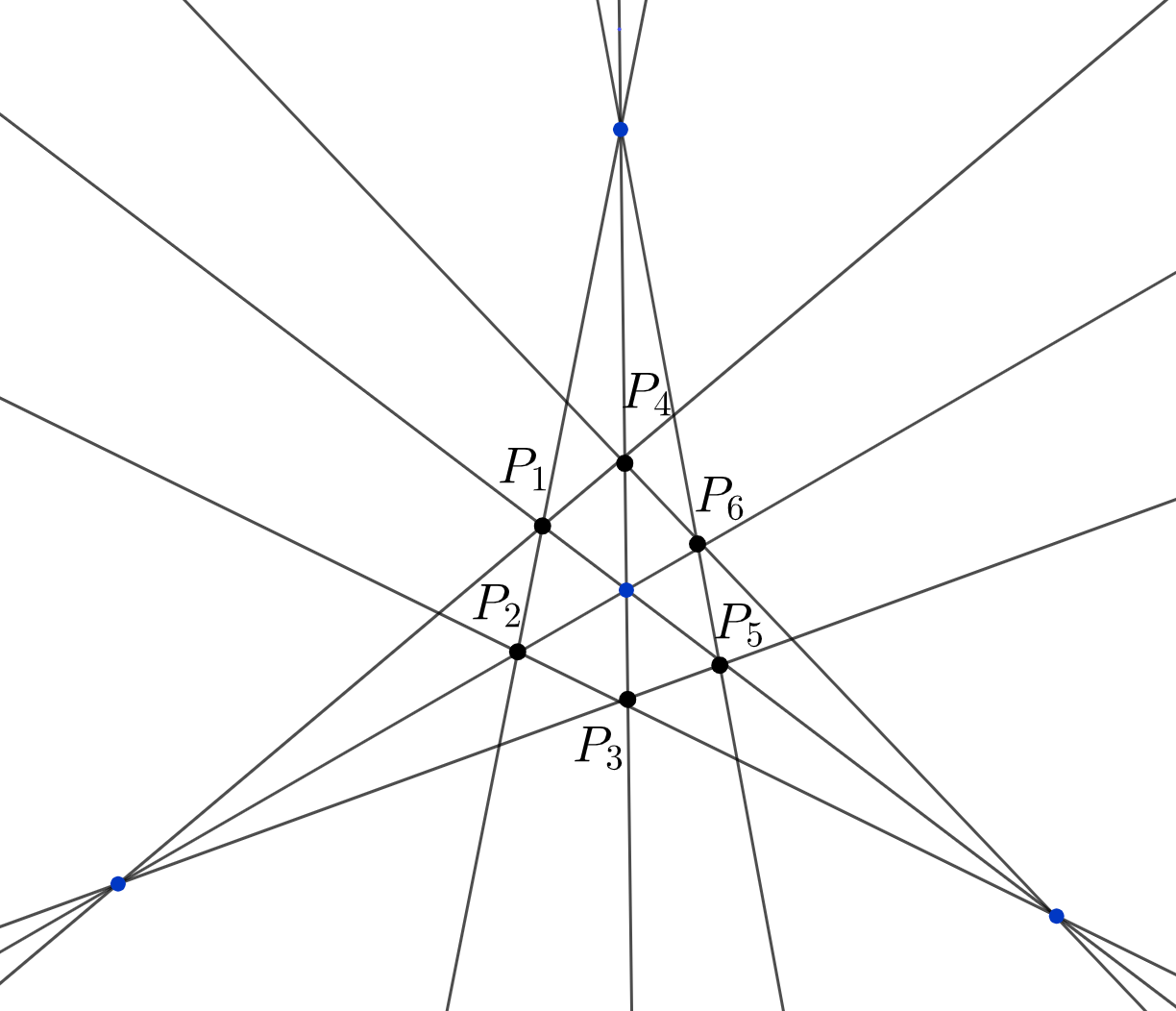} 
	\vspace{-0.33cm} $$
	\caption{Geometric representation of the strata of Type I, II, and IV in Table \ref{table:strata}
		\label{fig:configurations}}
\end{figure}

For our proof of Theorem~\ref{thm:numbersk=3}, we shall use Lemma \ref{lemma:stratifiedfibration}, 
here rewritten in the specific form
\begin{equation}\label{eq:chirho}
\chi(X(3,m+1)) \,\,\,\,= \,\,\, 
\chi(X(3,m))\cdot \chi(F_{X(3,m)})-\sum_{S\in \mathcal{S}} \chi(S) \,\cdot \rho(S),
\end{equation}
where we define
\begin{equation}\label{eq:rhoS}
\rho(S)  \quad := \sum_{S'\in \mathcal{S},S'\supseteq S} \!\! \mu(S,S')\cdot(\chi(F_{X(3,m)})-\chi(F_{S'})).
\end{equation}
For a fixed stratum $S$, the sum in \eqref{eq:rhoS} is over all $S'$ that contain $S$.
These form a subposet $\mathcal{P}_S$ (actually, a filter) of the poset $\mathcal{S}_{3,m}$.
It turns out that for many strata $S$, the factor $\rho(S)$ is zero.
This drastically simplifies the computation.
Let us illustrate this with an example.

\begin{example}
	Consider the map $\pi_{3,6}:X(3,7) \to X(3,6)$. Figure \ref{fig:posets} shows the subposets 
	for the two codimension three strata, of type IV and~V, together with the M\"obius function values at each node.
	The poset is ranked by codimension: zero for the top stratum, and three for the bottom stratum.
	Our aim is to compute $\rho(S)$. 	
	Let us look at the left poset $\mathcal{P}_S$. The codimension zero stratum $S'$ has $\chi(F_{S'})=\chi(F_{X(3,6)})$ by definition.
	The codimension one strata have $\chi(F_{X(3,m)})-\chi(F_{S'}) = 1$, by Example \ref{example:X36}. Looking at the strata of codimension two, we find that those of type III have $\chi(F_{X(3,6)}) - \chi(F_{S'}) = 3$, whereas those of type II have $\chi(F_{X(3,6)}) -\chi(F_{S'}) = 2$. Finally, the stratum of type IV has $\chi(F_{X(3,6)})-\chi(F_{S'}) = 4$. With this, one computes $\rho(S) = -2\cdot 0+3\cdot (1\cdot 1)+2\cdot 1+(-1)\cdot 3+3\cdot (-1\cdot 2)+1\cdot 4 = 0$.
\end{example}

\begin{figure}[h]
	\begin{subfigure}{0.45\linewidth}
		\centering \hspace*{-1.5cm}
		\begin{tikzpicture}[scale=1.2, transform shape]
		
		\node [label={[label distance=-0.2cm]90:\scriptsize{-2}}] (1) at (0, 5) {$\sbullet[0.8]$};
		\node [label={[label distance=-0.1cm,xshift = 0.15cm]180:\scriptsize{1}},label={[label distance=-0.2cm]0:\scriptsize{\color{blue}I}}] (2c) at (0.75,4) {$\sbullet[0.8]$};
		\node [label={[label distance=-0.1cm,xshift = 0.1cm]180:\scriptsize{2}},label={[label distance=-0.2cm]0:\scriptsize{\color{blue}I}}] (2d) at (2.25,4) {$\sbullet[0.8]$};
		\node [label={[label distance=-0.2cm,xshift = 0.05cm]180:\scriptsize{1}},label={[label distance=-0.2cm]0:\scriptsize{\color{blue}I}}] (2b) at (-0.75,4) {$\sbullet[0.8]$};
		\node [label={[label distance=-0.2cm,xshift = 0.05cm]180:\scriptsize{1}},label={[label distance=-0.1cm,xshift = -0.12cm]0:\scriptsize{\color{blue}I}}] (2a) at (-2.25,4) {$\sbullet[0.8]$};
		\node [label={[label distance=-0.2cm,xshift = 0.05cm]180:\scriptsize{-1}},label={[label distance=-0.2cm]0:\scriptsize{\color{blue}II}}] (3c) at (0.75,3) {$\sbullet[0.8]$};
		\node [label={[label distance=-0.1cm,xshift = 0.15cm]180:\scriptsize{-1}},label={[label distance=-0.2cm]0:\scriptsize{\color{blue}II}}] (3d) at (2.25,3) {$\sbullet[0.8]$};
		\node [label={[label distance=-0.2cm,xshift = 0.05cm]180:\scriptsize{-1}},label={[label distance=-0.0cm,xshift = -0.25cm]0:\scriptsize{\color{blue}II}}] (3b) at (-0.75,3) {$\sbullet[0.8]$};
		\node [label={[label distance=-0.2cm,xshift = 0.05cm]180:\scriptsize{-1}},label={[label distance=-0.0cm,xshift = -0.25cm]0:\scriptsize{\color{blue}III}}] (3a) at (-2.25,3) {$\sbullet[0.8]$};
		\node [label={[label distance=-0.2cm,xshift = -0.2cm,yshift = 0.3cm]270:\scriptsize{1}},label={[label distance=-0.1cm,xshift = -0.15cm,yshift = 0.3cm]350:\scriptsize{\color{blue}IV}}] (4) at (0,2) {$\sbullet[0.8]$};
		
		\draw [black,  opacity = 0.5, line width=0.1mm, shorten <=-8pt, shorten >=-8pt](1) -- (2a);
		\draw [black,  opacity = 0.5,  line width=0.1mm, shorten <=-10pt, shorten >=-8pt](1) -- (2b);
		\draw [black,  opacity = 0.5,  line width=0.1mm, shorten <=-10pt, shorten >=-8pt](1) -- (2c);
		\draw [black,  opacity = 0.5,  line width=0.1mm, shorten <=-8pt, shorten >=-8pt](1) -- (2d);
		\draw [black,  opacity = 0.5,  line width=0.1mm, shorten <=-8pt, shorten >=-8pt](2a) -- (3a);
		\draw [black,  opacity = 0.5,  line width=0.1mm, shorten <=-10pt, shorten >=-8pt](2a) -- (3b);
		\draw [black,  opacity = 0.5,  line width=0.1mm, shorten <=-10pt, shorten >=-10pt](2b) -- (3a);
		\draw [black,  opacity = 0.5,  line width=0.1mm, shorten <=-10pt, shorten >=-10pt](2b) -- (3c);
		\draw [black,  opacity = 0.5,  line width=0.1mm, shorten <=-8pt, shorten >=-8pt](2c) -- (3a);
		\draw [black,  opacity = 0.5,  line width=0.1mm, shorten <=-10pt, shorten >=-10pt](2c) -- (3d);
		\draw [black,  opacity = 0.5,  line width=0.1mm, shorten <=-8pt, shorten >=-8pt](2d) -- (3b);
		\draw [black,  opacity = 0.5,  line width=0.1mm, shorten <=-10pt, shorten >=-10pt](2d) -- (3c);
		\draw [black,  opacity = 0.5,  line width=0.1mm, shorten <=-8pt, shorten >=-8pt](2d) -- (3d);
		\draw [black,  opacity = 0.5,  line width=0.1mm, shorten <=-8pt, shorten >=-8pt](3a) -- (4);
		\draw [black,  opacity = 0.5,  line width=0.1mm, shorten <=-10pt, shorten >=-8pt](3b) -- (4);
		\draw [black,  opacity = 0.5,  line width=0.1mm, shorten <=-10pt, shorten >=-8pt](3c) -- (4);
		\draw [black,  opacity = 0.5,  line width=0.1mm, shorten <=-8pt, shorten >=-8pt](3d) -- (4);
		
		\end{tikzpicture}
	\end{subfigure}
	\begin{subfigure}{0.1\linewidth}
		\centering
		\begin{tikzpicture}[scale=1.3, transform shape]
		\node [label={[label distance=-0.2cm]90:\scriptsize{-6}}] (1) at (0, 5) {$\sbullet[0.8]$};
		\node [label={[label distance=-0.4cm,yshift = -0.21cm,xshift = -0.18cm]175:\scriptsize{2}},label={[label distance=-0.1cm,xshift = -0.1cm,yshift = 0.00cm]0:\scriptsize{\color{blue}I}}] (2a) at (-2.5,4) {$\sbullet[0.8]$};
		\node [label={[label distance=-0.4cm,yshift = -0.21cm,xshift = -0.18cm]175:\scriptsize{2}},label={[label distance=-0.1cm,xshift = -0.1cm]0:\scriptsize{\color{blue}I}}] (2b) at (-1.5,4) {$\sbullet[0.8]$};
		\node [label={[label distance=-0.4cm,yshift = -0.21cm,xshift = -0.18cm]175:\scriptsize{2}},label={[label distance=-0.1cm,xshift = -0.1cm]0:\scriptsize{\color{blue}I}}] (2c) at (-0.5,4) {$\sbullet[0.8]$};
		\node [label={[label distance=-0.25cm,yshift = -0.21cm,xshift = -0.27cm]87:\scriptsize{ 2}},label={[label distance=-0.1cm,xshift = -0.1cm]0:\scriptsize{\color{blue}I}}] (2d) at (0.5,4) {$\sbullet[0.8]$};
		\node [label={[label distance=-0.25cm,yshift = -0.21cm,xshift = -0.27cm]87:\scriptsize{ 2}},label={[label distance=-0.1cm,xshift = -0.1cm]0:\scriptsize{\color{blue}I}}] (2e) at (1.5,4) {$\sbullet[0.8]$};
		\node [label={[label distance=-0.25cm,yshift = -0.21cm,xshift = -0.27cm]:\scriptsize{ 2}},label={[label distance=-0.1cm,xshift = -0.1cm]0:\scriptsize{\color{blue}I}}] (2f) at (2.5,4) {$\sbullet[0.8]$};
		\node [label={[label distance=-0.2cm,xshift = 0.05cm]180:\scriptsize{-1}},label={[label distance=-0.1cm,xshift = -0.10cm]0:\scriptsize{\color{blue}II}}] (3a) at (-3,3) {$\sbullet[0.8]$};
		\node [label={[label distance=-0.2cm,xshift = 0.05cm]180:\scriptsize{-1}},label={[label distance=-0.1cm,xshift = -0.10cm]0:\scriptsize{\color{blue}III}}] (3b) at (-2,3) {$\sbullet[0.8]$};
		\node [label={[label distance=-0.25cm]180:\scriptsize{-1}},label={[label distance=-0.05cm,xshift = -0.16cm]0:\scriptsize{\color{blue}III}}] (3c) at (-1,3) {$\sbullet[0.8]$};
		\node [label={[label distance=-0.3cm,xshift=-0.03cm]180:\scriptsize{-1}},label={[label distance=-0.2cm,xshift = -0.05cm]0:\scriptsize{\color{blue}III}}] (3d) at (0,3) {$\sbullet[0.8]$};
		\node [label={[label distance=-0.2cm,xshift = 0.08cm]180:\scriptsize{-1 }},label={[label distance=-0.2cm,xshift = -0.03cm]0:\scriptsize{\color{blue}III}}] (3e) at (1,3) {$\sbullet[0.8]$};
		\node [label={[label distance=-0.2cm,xshift=0.10cm]180:\scriptsize{-1 }},label={[label distance=-0.2cm,xshift = -0.05cm]0:\scriptsize{\color{blue}II}}] (3f) at (2,3) {$\sbullet[0.8]$};
		\node [label={[label distance=-0.07cm,xshift = 0.2cm]180:\scriptsize{-1 }},label={[label distance=-0.2cm,xshift = -0.05cm]0:\scriptsize{\color{blue}II}}] (3g) at (3,3) {$\sbullet[0.8]$};
		\node [label={[label distance=-0.2cm,xshift = -0.2cm,yshift = 0.3cm]270:\scriptsize{1}},label={[label distance=-0.2cm,xshift = -0.05cm,yshift= 0.3cm]330:\scriptsize{\color{blue}V}}] (4) at (0,2) {$\sbullet[0.8]$};
		
		\draw [black,  opacity = 0.5,  line width=0.1mm, shorten <=-10pt, shorten >=-10pt](1) -- (2a);
		\draw [black,  opacity = 0.5,  line width=0.1mm, shorten <=-10pt, shorten >=-10pt](1) -- (2b);
		\draw [black,  opacity = 0.5,  line width=0.1mm, shorten <=-10pt, shorten >=-10pt](1) -- (2c);
		\draw [black,  opacity = 0.5,  line width=0.1mm, shorten <=-10pt, shorten >=-10pt](1) -- (2d);
		\draw [black,  opacity = 0.5,  line width=0.1mm, shorten <=-10pt, shorten >=-10pt](1) -- (2e);
		\draw [black,  opacity = 0.5,  line width=0.1mm, shorten <=-10pt, shorten >=-10pt](1) -- (2f);
		\draw [black,  opacity = 0.5,  line width=0.1mm, shorten <=-10pt, shorten >=-10pt](2a) -- (3a);
		\draw [black,  opacity = 0.5,  line width=0.1mm, shorten <=-10pt, shorten >=-10pt](2a) -- (3b);
		\draw [black,  opacity = 0.5,  line width=0.1mm, shorten <=-10pt, shorten >=-10pt](2a) -- (3c);
		\draw [black,  opacity = 0.5,  line width=0.1mm, shorten <=-10pt, shorten >=-10pt](2b) -- (3a);
		\draw [black,  opacity = 0.5,  line width=0.1mm, shorten <=-10pt, shorten >=-10pt](2b) -- (3d);
		\draw [black,  opacity = 0.5,  line width=0.1mm, shorten <=-10pt, shorten >=-10pt](2b) -- (3e);
		\draw [black,  opacity = 0.5,  line width=0.1mm, shorten <=-10pt, shorten >=-10pt](2c) -- (3b);
		\draw [black,  opacity = 0.5,  line width=0.1mm, shorten <=-10pt, shorten >=-10pt](2c) -- (3d);
		\draw [black,  opacity = 0.5,  line width=0.1mm, shorten <=-10pt, shorten >=-10pt](2c) -- (3f);
		\draw [black,  opacity = 0.5,  line width=0.1mm, shorten <=-10pt, shorten >=-10pt](2d) -- (3b);
		\draw [black,  opacity = 0.5,  line width=0.1mm, shorten <=-10pt, shorten >=-10pt](2d) -- (3e);
		\draw [black,  opacity = 0.5,  line width=0.1mm, shorten <=-10pt, shorten >=-10pt](2d) -- (3g);
		\draw [black,  opacity = 0.5,  line width=0.1mm, shorten <=-10pt, shorten >=-10pt](2e) -- (3c);
		\draw [black,  opacity = 0.5,  line width=0.1mm, shorten <=-10pt, shorten >=-10pt](2e) -- (3d);
		\draw [black,  opacity = 0.5,  line width=0.1mm, shorten <=-10pt, shorten >=-10pt](2e) -- (3g);
		\draw [black,  opacity = 0.5,  line width=0.1mm, shorten <=-10pt, shorten >=-10pt](2f) -- (3c);
		\draw [black,  opacity = 0.5,  line width=0.1mm, shorten <=-10pt, shorten >=-10pt](2f) -- (3e);
		\draw [black,  opacity = 0.5,  line width=0.1mm, shorten <=-10pt, shorten >=-10pt](2f) -- (3f);
		\draw [black,  opacity = 0.5,  line width=0.1mm, shorten <=-10pt, shorten >=-10pt](3a) -- (4);
		\draw [black,  opacity = 0.5,  line width=0.1mm, shorten <=-10pt, shorten >=-10pt](3b) -- (4);
		\draw [black,  opacity = 0.5,  line width=0.1mm, shorten <=-10pt, shorten >=-10pt](3c) -- (4);
		\draw [black,  opacity = 0.5,  line width=0.1mm, shorten <=-10pt, shorten >=-9pt](3d) -- (4);
		\draw [black,  opacity = 0.5,  line width=0.1mm, shorten <=-10pt, shorten >=-10pt](3e) -- (4);
		\draw [black,  opacity = 0.5,  line width=0.1mm, shorten <=-10pt, shorten >=-10pt](3f) -- (4);
		\draw [black,  opacity = 0.5,  line width=0.1mm, shorten <=-10pt, shorten >=-10pt](3g) -- (4);				
		\end{tikzpicture}
	\end{subfigure}
	\caption{The posets $\mathcal{P}_S$ for the codimension three strata in $X(3,6)$. The numbers in black are the values of the M\"obius function.
		The types of the strata are shown in blue.}
	\label{fig:posets}
\end{figure}
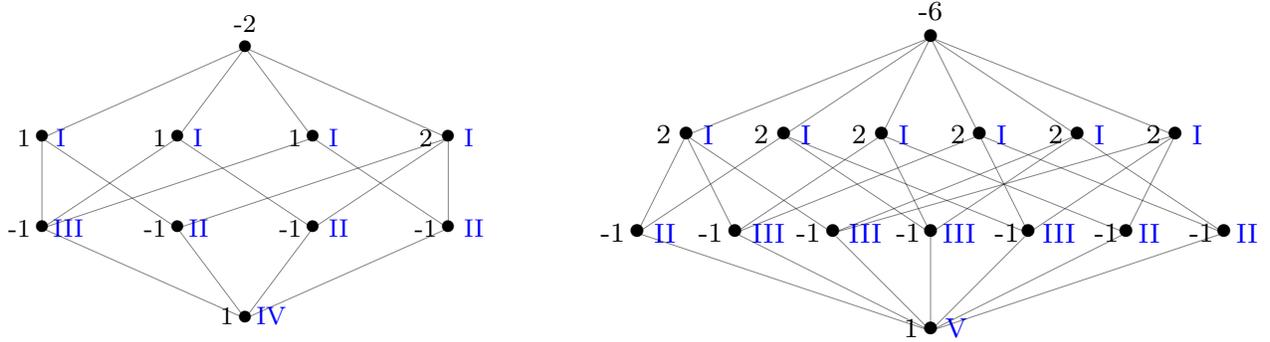

This phenomenon generalizes to other strata and to all spaces $X(3,m)$. Indeed, the only strata $S\in \mathcal{S}_{3,m}$ for which $\rho(S)\ne 0$ are those whose lines meet in only one extra special point apart from the original ones.  This is illustrated 
 in \cite[Figure~8]{CUZ}.
More precisely, for each stratum $S$, let $[P]=[P_1,\dots,P_m]$ be a general element in $S$, and for each $h\geq 3$ denote by $n_h(S)$ the number of points, apart from $P_1,\dots,P_m$, where exactly $h$ lines of the discriminantal arrangement $\mathcal{B}(P)$ meet. The next result proves the conjecture in~\cite[\S 6]{CUZ}: 

\begin{theorem}\label{thm:rhostrata}
Every stratum $S \in \mathcal{S}_{3,m}$ satisfies $\,\rho(S) \in \{-1,0,+1\}$. More precisely:
	\begin{enumerate}
		\item[(a)]  If $n_h(S)=1$ for some $h$ and $n_i(S)=0$ for all $i \ne h$, then $\rho(S) = (-1)^{h-1}$. 
		\item[(b)] In all other cases, we have $\rho(S)=0$.
	\end{enumerate}
\end{theorem} 

Item (b) is the punchline: a stratum $S$ does not contribute to the Euler characteristic unless
 it looks like \cite[Figure~8]{CUZ}.
To prove this, we introduce some further notation. For any stratum $S \in \mathcal{S}_{3,m}$ let $\sigma(S) \colon\!\!\!= \chi(F_{X(3,m)}) - \chi(F_S)$. When the discriminantal arrangement in the fiber $F_S$ over $S$ is realizable over $\mathbb{R}$, this denotes the difference between the numbers of bounded regions in 
$\mathcal{B}(3,m)$ and $F_S$. Applying the M\"obius inversion formula to \eqref{eq:rhoS} yields
\begin{equation}\label{eq:rhosigma}
\sigma(S) \,\,=\,\, \rho(S) \,\,+ \sum_{S'\in \mathcal{P}_S\setminus \{ S\}} \rho(S').
\end{equation}
We start by computing the function $\sigma(S)$
 in terms of the counting function $n_h(S)$:

\begin{lemma}\label{lemma:sigma}
	Every stratum $S\in \mathcal{S}_{3,m}$ satisfies
	$\, \sigma(S) \,=\, \sum_{h=3}^{\infty} \binom{h-1}{2}  \cdot n_h(S)$.
\end{lemma}

\begin{proof}
	Let $\mathcal B$ be the discriminantal arrangement associated to the fiber~$F_S$.
	Let $\mathcal{B}(3,m)$ be as in Section \ref{section3}.
	Its characteristic polynomial is $\chi_{\mathcal B(3,m)}=t^2- \left ( \binom{m}{2} -1 \right) t+b_2(\mathcal{B}(3,m))$ where $b_2(\mathcal{B}(3,m))$ is the second Betti number of $\mathcal{B}(3,m)$.
	Using \eqref{eq:char_poly}, the characteristic polynomial of the 
	special arrangement~$\mathcal{B}$ equals
	\[
	\chi_{\cal B}(t)\,\,=\,\,t^2 -\left ( \binom{m}{2} -1 \right)t + b_2(\mathcal{B}(3,m)) -\sum_{h=3}^{\infty} n_h(S)\binom{h}{2}+ \sum_{h=3}^{\infty} n_h(S)(h-1),
	\]
	since every additional intersection point of $h$ lines in $\mathcal B$ removes $\binom{h}{2}$ many generic codimension two elements of $L(\mathcal B(3,m))$ and adds a summand with M\"obius value $h-1$. 	Evaluating these two polynomials at $t=1$ and taking their difference yields the claim.	
\end{proof}

We now show that $\sigma$ is monotonic with respect to the poset $\mathcal{S}$ in the following sense.

\begin{lemma}\label{lemma:sigmaincrease}
	Let $S',S \in \mathcal{S}_{3,m}$ be two strata such that $S' \supsetneq S$. Then $\sigma(S') < \sigma(S)$.
\end{lemma}

\begin{proof}
	The assertion is equivalent to $\chi(F_{S'}) =\chi_{\mathcal B_{S'}}(1)> \chi(F_{S})=\chi_{\mathcal B_{S}}(1)$ where $\mathcal B_{S}$ and $\mathcal B_{S'}$ are the discriminantal arrangements
	corresponding to the fibers $F_S$ and $F_{S'}$, respectively.
	Since each arrangement has the same number of lines, it suffices to prove $b_2(\mathcal B_{S'}) > b_2(\mathcal B_{S})$. 
	
	We will establish a general statement, namely
a strict inequality between the Betti numbers of two arrangements where one is a strict weak image of the other in the sense of matroid theory.
	Let $\mathcal A=\{H_1,\dots,H_n\}$ and $\mathcal{A}'=\{H_1',\dots,H_n'\}$ be two affine arrangements of lines in $\CC^2$ with $n\ge 3$.
We claim that $b_2(\mathcal A) < b_2(\mathcal A')$ holds under the following two assumptions:
	\begin{enumerate}
		\item[(i)]  If for $I\subseteq [n]$  the lines $\{H_i'\}_{i \in I}$ meet in $\mathcal A'$ then the lines $\{H_i\}_{ i \in I}$ also meet in $\mathcal{A}$, 
		\item[(ii)]The lines $H_1,H_2,H_3$ intersect but the lines $H_1',H_2',H_3'$ do not intersect.
	\end{enumerate}
		We prove this claim by induction on $n$.
	If $n=3$ then the statement is trivial since (ii) forces $\mathcal A'$ to consist of three generic lines and $\mathcal A$ of three concurrent lines, so $b_2(\mathcal A) = 2< 3=b_2(\mathcal A')$.
	
	Now fix some $n>3$.
	The deletion--restriction relation for the characteristic polynomial (cf.~\cite[Corollary 2.57]{OT}) implies $b_2(\mathcal A) = b_2(\mathcal A\setminus\{H_n\}) + |\mathcal A^{H_n}|$ where $\mathcal A^{H_n}$ is the restriction of $\mathcal A$ to $H_n$. The analogous relation holds for $\mathcal A'$.
	Since removing the last hyperplane preserves both assumptions (i) and (ii), we have $b_2(\mathcal A\setminus\{H_n\}) < b_2(\mathcal A'\setminus\{H_n'\})$ by induction.
	Moreover, assumption (i) implies $|\mathcal A^{H_n}| \le  |\mathcal A'^{H_n'}|$.
	So in total this proves $b_2(\mathcal A) < b_2(\mathcal A')$.
\end{proof}

\begin{proof}[Proof of Theorem \ref{thm:rhostrata}]
Note that $\rho(X(3,m)) = 0$ by definition. Let $S$ be a stratum with the property that $n_h(S)=1$ and all other are zero. Observe that the only $S' \subsetneq X(3,m)$ strictly containing $S$ are the $\binom{h}{j}$ strata $S'$ with $n_j(S')=1$ and all others zero for $j=3,\ldots h-1$. If $h=3$ then $S$ is a codimension one stratum and is only strictly contained in $X(3,m)$, so Lemma \ref{lemma:sigma} and formula \eqref{eq:rhosigma} show that $\rho(S)=1$. We now rewrite \eqref{eq:rhosigma} via Lemma \ref{lemma:sigma} as 
	\[ \binom{h-1}{2}\,\, =\,\, \rho(S) + \binom{h}{3} - \binom{h}{4} + \dots + (-1)^{h} \binom{h}{h-1} \]
	by induction on $h$. This shows that $\rho(S) = (-1)^{h-1}$.
	
	We prove part (b) by induction on $\sigma(S)$. The base case $\sigma(S)=0$ corresponds to the stratum $X(3,m)$, for which we know already that $\rho(S)=0$. For the induction step, let $S \in \mathcal{S}_{3,m}$ be another stratum that is not one of those considered in part (a). For all strata $S'\in \mathcal{P}_S\setminus \{S\}$ we have $\sigma(S') < \sigma(S)$ by Lemma \ref{lemma:sigmaincrease}.
	The induction hypothesis shows that $\rho(S')=0$ for all strata that are not of the type of part (a). Thus, we have 
	\[
	\sigma(S) \,\,=\,\, \sum_{h=3}^{\infty} \binom{h-1}{2} \cdot n_h(S)  = \rho(S) + \sum_{h=3}^{\infty} N_h(S) \cdot (-1)^{h-1} 
	\]
	where $N_h(S)$ is the number of strata $S' \in \mathcal{P}_S\setminus\{S\}$ like those in part (a). This number is given directly by $N_h(S) = \sum_{k=3}^{\infty} n_k(S) \cdot \binom{k}{h}$. We conclude that
	$$
	\sum_{h=3}^{\infty} \binom{h\! - \! 1}{2} \cdot n_h(S)  \,\, = \,\,
	  \rho(S) + \sum_{k=3}^{\infty} n_k(S)\cdot  \sum_{h=3}^{\infty} (-1)^{h-1} \binom{k}{h} 
	 \,\, = \,\, \rho(S) + \sum_{k=3}^{\infty} n_k(S) \cdot \binom{k \! - \! 1}{2}  .
$$
	This implies $\rho(S)=0$ for all strata $S$ other than those in case (a).
\end{proof}

Equipped with Theorem \ref{thm:rhostrata}, we can prove Theorem \ref{thm:numbersk=3} with a few calculations.

\begin{proof}[Proof of Theorem \ref{thm:numbersk=3}]
	We want to use the formula \eqref{eq:chirho}. By Theorem \ref{thm:rhostrata}, we need to consider only those strata $S\in \mathcal{S}_{3,m}$ where $n_h(S)=1$ for a certain $h$ and $n_i(S)=0$ for all $i\ne h$. The number of these strata in $X(3,m)$ is $\frac{1}{h!}\binom{m}{2}\cdot \binom{m-2}{2} \dots \binom{m-2h+2}{2}$. Furthermore, all these strata are combinatorially equivalent, so we write $\chi(3,m;h)$ for the Euler characteristic of any of them. Using Theorem \ref{thm:rhostrata} we rewrite the formula \eqref{eq:chirho} as follows:
$$ \chi(X(3,m+1)) \,\,=\,\, \chi(X(3,m))\cdot \chi(F_{X(3,m)}) \,+\, \sum_{h=3}^{\infty}{\textstyle \frac{(-1)^h}{h!}\binom{m}{2}\cdot \binom{m-2}{2} \cdots \binom{m-2h+2}{2} }\cdot \chi(3,m;h).  $$
	In particular, when $m=6,7$ the formula becomes
	\begin{align*}
	\chi(X(3,7)) &= \chi(X(3,6))\cdot \chi(F_{X(3,6)})- 15 \cdot \chi(3,6;3), \\
	\chi(X(3,8)) &= \chi(X(3,7))\cdot \chi(F_{X(3,7)}) - 105 \cdot \chi(3,7;3).
	\end{align*}	
		The Euler characteristic $\chi(X(3,6)) = 26$ is obtained in Example \ref{example:X36}, whereas the bounded chamber counts $\chi(F_{X(3,6)})=42$ and $\chi(F_{X(3,7)}) = 101$ are taken from Table \ref{tab:parallel}. To compute $\chi(3,m;3)$, we note that corresponding stratum in
	$X(3,m)$ is isomorphic to a matroid stratum of codimension
	$3$ in $X(3,m+1)$, obtained by requiring that the new point $P_{m+1}$
	lies on the three special lines. Those matroid strata
	have Euler characteristic $\chi(3,6;3) = -12$ for $m=6$ and $\chi(3,7;3) = -568$. 
	This is proved either geometrically, or by computer algebra.
	Note that these strata are identified uniquely in
	Theorem \ref{thm:matroid_ml}, namely for $m=6$   under $\,3_5[ \ldots, 12]$,
	and for $m=7$ under $\,5_7[\ldots, 568]$. In conclusion, we can write $	\chi(X(3,7)) = 26 \cdot 42 + 15 \cdot 12 = 1\,272$ and $ 
	\chi(X(3,8)) = 1\,272 \cdot 101 + 105 \cdot 568  = 188\,112$.
	If instead $m=8$ the formula becomes
	\[ \chi(X(3,9)) = \chi(X(3,8)) \cdot \chi(F_{X(3,8)}) - 420 \cdot \chi(3,8;3) + 105 \cdot \chi(3,8;4). \]
	Table \ref{tab:parallel} shows that $\chi(F_{X(3,8)})=205$. The numbers $\chi(3,8;3)=-81\,040$ and $\chi(3,8;4)=18\,768$ can be computed using \texttt{HomotopyContinuation.jl}~\cite{HCjl} as the ML degrees of the corresponding strata.
	Note that $18\,768$ appears in Theorem~\ref{thm:matroid_ml} as the ML degree of the rank $3$ matroid on $9$ elements corresponding to four lines meeting in a point. The number $81\,040$ does not appear in this theorem.
	However, $81\,040 = 99\,808  - 18\,768$  where $99\,808$ is the ML degree of the matroid on $9$ elements where three lines meet in a point. Thus, the contribution above stems from this matroid strata where one needs to remove the locus of four concurrent lines. In conclusion, we have $\chi(X(3,9)) = 188\,112 \cdot 205 + 420 \cdot 81\,040 + 105 \cdot 18\,768 = 74\,570\,400$.\end{proof}

As a proof of concept, we show how this story extends to nonuniform matroids. 

\begin{example}
\label{ex:NonUniformMLDegree}
	Let $M$ be the matroid of rank $3$ on $9$ elements with one nonbasis $789$. The 
	map $X(M) \to X(3,8)$ which forgets the ninth point is a surjection. The generic fiber  is no longer the complement of $\mathcal B(3,8)$, but rather its restriction to the line $\overline{P_7P_8}$.
	That restricted arrangement has $15$ bounded regions. The nongeneric fiber over $S_{(ij)(kl)(rs)}$ has $14$ bounded regions provided $(78) \in \{(ij),(kl),(rs)\}$, and $15$ otherwise. There are $45$ strata of the form $S_{(ij)(kl)(78)}$. Similarly, the fiber over a stratum of four concurrent lines has $13$ bounded regions provided $\overline{P_7P_8}$ is one of those lines. There are $15$ such codimension $2$ strata. This gives
	\[\chi(X(M)) \,\,=\,\, 188\,112 \cdot 15 \,+\, 45 \cdot 81\,040 \,+\, 15 \cdot 18768 \,\,=\,\, 6\,750\,000,
	\]
	as a decomposition of $6\,750\,000$, verifying the entry $\,9_1[6750000_1]\,$ in Theorem \ref{thm:matroid_ml}.  \end{example}

\section{Eight points in 3-space} 
\label{section6}

We now turn to configurations in $\PP^3$. For $k=4$ and $m \leq 7$, 
we can apply Grassmann duality, which shows that $X(4,m) = X(m-4,m)$.
Hence the ML degrees for few points in $\PP^3$ are
$$
 |\chi(X(4,6))| = |\chi(X(2,6))| = 6 \quad {\rm and} \quad \chi(X(4,7)) = \chi(X(3,7)) = 1\,272.
$$
This section is devoted to the $9$-dimensional very affine variety $X(4,8)$.
Here is our result:

\begin{theoremstar} \label{conj:chi48}
The Euler characteristic of $\,X(4,8)$ equals $-\,5\,211\,816$.
\end{theoremstar}

The number $\,5\,211\,816\,$ is new.
Unlike the ML degrees in Theorem \ref{thm:numbersk=3},
it did not yet appear in the physics literature.
Also, Lam's finite field method (Appendix \ref{appendixA})
 did not yield this number.
Thus far, the $(4,8)$ case was out of reach for 
all available techniques, in spite of the progress in
\cite[\S 6]{CUZ}. Our result proves the
conjecture stated tacitly in
\cite[Table~2]{CUZ}.

Our derivation of Theorem* \ref{conj:chi48} rests on numerical
computations with the software {\tt HomotopyContinuation.jl} \cite{HCjl}.
However, the methodology is closely related to the topological approach seen
in previous sections. We fully exploit a specific likelihood degeneration,  namely the
soft limits \cite{CUZ}, to analyze $X(4,8)$ in a way that is similar to
Lemma \ref{lemma:stratifiedfibration}

We start by describing our computational setup. For each quadruple $(i,j,k,\ell)$, 
where $1 \leq i < j < k < \ell \leq 8$, let $p_{ijk\ell}$ be the determinant of the
corresponding $4 \times 4$ submatrix of 
\begin{align} \label{eq:48matrix}
M_{4,8} \,\,=\,\, \begin{small} \begin{bmatrix}
0 & 0 & 0 & 1 & 1 & 1 & 1 & 1 \\
0 & 0 &-1 & 0 & 1 & x_1 & x_2 & x_3 \\ 
0 & 1 & 0 & 0 & 1 & y_1 & y_2 & y_3 \\
-1 & 0 & 0 & 0 & 1 & z_1 & z_2 & z_3
\end{bmatrix}. \end{small}
\end{align}
 Let ${\cal I}$ be the set of all such quadruples with $k \geq 4$ and $\ell \geq 6$. We write $(x,y,z)$ for the nine variables appearing in \eqref{eq:48matrix}. Note that $ X(4,8)$ is the complement of the union of hypersurfaces ${\cal H} = \bigcup_{(i,j,k,\ell) \in {\cal I}} V(p_{ijk\ell}) \subset \CC^9$. This is a very affine variety in $(\CC^*)^{62}$ with parametrization $\CC^9 \setminus {\cal H} \rightarrow  (\CC^*)^{62}$ given by 
$(x,y,z) \mapsto (p_{ijk\ell}(x,y,z))_{(i,j,k,\ell) \in {\cal I}}.$
By  \cite[Theorem~1]{Huh13}, the quantity $|\chi(X(4,8))|$ is the number of critical points of the log-likelihood function 
$${\cal L}_{4,8}(x,y,z) \,\,\, = \sum_{(i,j,k,\ell) \in {\cal I}} \!\!\! u_{ijk\ell} \cdot \log(p_{ijk\ell}(x,y,z)).$$
Here we count critical points in $\CC^9 \backslash {\cal H}$, for generic data $u_{ijk\ell} \in \CC$.
In other words, we seek
   the number of solutions to the following system of 9 rational function equations in 9 unknowns:
\begin{align} \label{eq:critpt48}
\frac{\partial { \cal L}_{4,8}}{\partial x_i}\, = \,\frac{\partial {\cal L}_{4,8}}{\partial y_i} \,=
\, \frac{\partial {\cal L}_{4,8}}{\partial z_i}\, =\, 0 \qquad {\rm for} \quad i = 1,2, 3.
\end{align}
In particle physics, these are the \emph{scattering equations} for $X(4,8)$ in the CEGM model \cite{CEGM}. The projection of the likelihood correspondence \cite{HS, VS}
to the space of data is a branched covering of degree $|\chi(X(4,8))|$. For generic complex 
numbers $u_{ijk\ell}$, the fiber can be computed using the command \texttt{monodromy\_solve} in \texttt{HomotopyContinuation.jl}, as explained in \cite[\S 3]{ST}. 
In principle, we can use the \texttt{certify} command \cite{BRT}
to give a proof of the inequality 
$$|\chi(X(4,8))|\,\, \geq \,\,5\,211\,816. $$
In practise, we missed $218$ solutions.
In our run,
the method quickly certified that $5\,211\,598$ paths correspond to distinct true solutions, giving a proof that 
$|\chi(X(4,8))| \geq 5\,211\,598$.

The main idea in this section is to follow up this brute force monodromy computation by a likelihood degeneration of \eqref{eq:critpt48} to study $X(4,8)$ in a more structured way. In particular, the degeneration will help us to decompose the number $|\chi(X(4,8))|$ into positive summands, much like what we did in the Section \ref{section5}. We keep assuming that the data $u_{ijk\ell}$ are generic. 

We introduce a parameter $t$ into the log-likelihood function ${\cal L}_{4,8}$ by setting
\begin{equation}
\label{eq:L48withparameter}
 \tilde{\cal L}_{4,8}(x,y,z,t) \,\,\,= \sum_{\substack{ (i,j,k,\ell) \in {\cal I} \\ \ell < 8}}
 \!\!\!\! u_{ijk\ell} \cdot \log (p_{ijk\ell}(x,y,z)) 
\, \,\,+\!\! \sum_{(i,j,k,8) \in {\cal I}}
\!\!\!\! u_{ijk8} \cdot t \cdot \log (p_{ijk8}(x,y,z)).
 \end{equation}
The limit for $t \rightarrow 0$ is the soft limit in \cite{CUZ}. Taking partial derivatives, we obtain rational function equations in the unknowns $x,y,z$ with coefficients in the rational function field $\CC(t)$:
\begin{align} \label{eq:critpt48t}
\frac{\partial \tilde{\cal L}_{4,8}(x,y,z;t)}{\partial x_i} \,\,=\,\,
 \frac{\partial \tilde{\cal L}_{4,8}(x,y,z;t)}{\partial y_i} \,\,=\,\,
  \frac{\partial \tilde{\cal L}_{4,8}(x,y,z;t)}{\partial z_i} \,\,=\,\, 0 \quad {\rm for} \,\,\,\ i = 1, 2, 3.
\end{align}
There are $|\chi(X(4,8))|$ solutions $(\hat{x}(t),\hat{y}(t),\hat{z}(t))$ over 
the field of Puiseux series   $K = \CC \{\!\{t\}\!\}$.
We are interested in computing the valuations of the Pl\"ucker coordinates
for these solutions. In other words, suppose we knew these solutions,
and suppose we were to
substitute them into the $4 \times 8$ matrix (\ref{eq:48matrix}). Each of its 
 $4 \times 4$ minors is then a Puiseux series $\hat{p}_{\bullet}(t)$, and we could 
 consider the lowest order exponent  of that series.
This is the $t$-adic valuation of $\hat{p}_{\bullet} = \hat{p}_{\bullet}(t)$, denoted ${\rm val}_t(\hat{p}_{\bullet})$.
The result of this process would be the tropical Pl\"ucker vector
\begin{align} \label{eq:valpatterns}
\hat{q} \,\,=\,\,  \bigl(\,{\rm val}_t \bigl(\,p_{ijk\ell}(\hat{x}(t),\hat{y}(t),\hat{z}(t))\,\bigr)\,
\bigr)_{1\leq i < j < k < \ell \leq 8} \,\,\in\,\, \QQ^{70}.
\end{align} 

Recall from (\ref{eq:asthequotient}) that $X(4,8)$ is the quotient
of ${\rm Gr}(4,8)^\circ$ by the action of the torus $(K^*)^8$.
On the tropical side, where the vector $\hat{q}$ lives, this corresponds to
an additive action of $\RR^8$ on $\RR^{70}$. We take the quotient of this
additive action by setting to $0$ the eight coordinates not in $\mathcal{I}$.
Thus our choice of $62$ coordinates is compatible with tropicalization (cf. \cite{MS}),
and we obtain $\,{\rm trop}(X(4,8))\,$ as a $9$-dimensional
pointed fan in $\RR^{62}$, with coordinates indexed by $\mathcal{I}$.

Now, here is the punchline: we cannot compute solutions over $K$,
and we have no access to the Puiseux series in the argument of
${\rm val}_t$ in (\ref{eq:valpatterns}).
Instead, we carry out floating point computations over $\RR$. This will
give us enough information to identify the coordinates of~$\hat{q}$.
Indeed, from the point of view of complex geometry, the equations \eqref{eq:critpt48t} define an affine curve
$$ C \,\,\subset \,\,X(4,8) \times \CC \,\,\subset \,\,(\CC^*)^{62} \times \CC, $$
where the second factor is the line with coordinate $t$.
 The vectors $\hat{q}$ in \eqref{eq:valpatterns}, which will be called \emph{tropical critical points} in 
 Section \ref{section7},    span the rays in the partial tropical curve 
 $$ {\rm trop}(C)\,\, \subset \,\,{\rm trop}(X(4,8)) \times \mathbb{R}_{\geq 0}
 \,\, \subset \,\,\RR^{62} \times \RR_{\geq 0}.
 $$
 The solution $\hat{q} = 0$  represents classical solutions $(x(t),y(t),z(t))$ which, in the soft limit $t \rightarrow 0$, converge in $X(4,8)$. The corresponding ray in ${\rm trop}(C)$ is $ \rho_0 = \RR_{\geq  0} \cdot (0,\ldots,0,1)$.

\begin{proposition} \label{prop:regsols48}
The ray $\,\rho_0 = \RR_{\geq  0} \cdot (0,\ldots,0,1)\,$ has multiplicity $\,2\,363\,376\,$ in $\,  {\rm trop}(C)$.
\end{proposition}

\begin{proof}
The multiplicity of $\rho_0$ is the number of classical solutions converging in $X(4,8)$ for $t \rightarrow 0$. In this limit, the equations for $i = 1,2$ in \eqref{eq:critpt48t} are the likelihood equations for $X(4,7)$. Plugging any solution of these equations into those for $i = 3$, we find, up to division by $t$, the likelihood equations for the complement in $\CC^3$ of the discriminantal arrangement ${\cal B}(4,7)$. 
The ML degree of that arrangement complement is $1\,858$, as seen in Table \ref{tab:parallel}.
This gives the formula $\,\chi(X(4,7)) \cdot |\chi(\CC^3 \setminus {\cal B}(4,8)) |= 1\,272 \cdot 1\,858\,$
 for the multiplicity of $\rho_0$.
\end{proof}

The number $2\,363\,376$ is the number of \emph{regular solutions} in \cite{CUZ}. In the spirit of
Sections~\ref{section2} and
\ref{section5}, it is the contribution to $\chi(X(4,8))$ coming from the dense stratum in $X(4,7)$.
The nonzero tropical critical points $\hat q$ 
correspond to the \emph{singular solutions} in \cite{CUZ}. 
For $t \rightarrow 0$, these curves
move to the boundary of $X(4,8)$. 
The following result verifies a conjecture~in~\cite{CUZ}.

\begin{theoremstar} \label{conj:48}
There are $3\,150$ distinct nonzero tropical critical points 
$\hat{q}$. All of them are given by $\{0,1\}$-vectors in $\RR^{70}$ and they come in 
$7$ combinatorial types, summarized in Table~\ref{tab:combtypes48}. 
\end{theoremstar}

\begin{table}[]
\centering
\begin{tabular}{c|c|c}
Type & representative  ($(i,j,k,\ell)$ for which ${\rm val}_t(p_{ijk\ell}) = 1$)                                                                    & \# configurations \\ \hline
I    & (1,4,7,8),(2,5,7,8),(3,6,7,8)                                                 & 105                      \\
II   & (1,2,3,8),(3,4,5,8),(5,6,1,8),(2,4,6,8)                                     & 210                      \\
III  & (1,2,3,8),(3,4,5,8),(5,6,7,8),(2,4,6,8)                                     & 1260                     \\
IV   & (1,4,7,8),(2,5,7,8),(3,6,7,8),(1,2,3,8),(4,5,6,8)                         & 420                      \\
V    & (1,2,3,8),(1,4,5,8),(1,6,7,8),(2,4,6,8),(2,5,7,8)                         & 630                      \\
VI   & (1,2,3,8),(1,4,5,8),(1,6,7,8),(2,4,6,8),(2,5,7,8),(3,4,7,8)             & 210                      \\
VII  & (1,2,3,8),(1,2,4,8),(1,2,5,8),(1,2,6,8),(1,2,7,8),(3,4,5,8),(3,6,7,8) & 315                     
\end{tabular}
\caption{Combinatorial types contributing singular solutions to \eqref{eq:critpt48t}. The second column shows a representative for each type indicating which Pl\"ucker coordinates have valuation 1. The third column shows the cardinality of the orbit of the action of $S_7$ on the first 7 indices. }
\label{tab:combtypes48}
\end{table}

We here assume that the $u_{ijk\ell}$ are generic complex numbers.
The result is derived by  computations with {\tt HomotopyContinuation.jl}.
The $t$-adic order of each numerical solution was
found using 
 Algorithm \ref{algo:computetropicalcriticalpoints}.
 The nonzero tropical critical points $\hat{q}$ span the rays in ${\rm trop}(C)$, other than $\rho_0$, whose generator has a positive last coordinate. The ML degree equals
 \begin{equation} \label{eq:decompchi48}
|\chi(X(4,8))| \,\,= \,\,\, 2\,363\,376  \,+\, A_{\text I} \cdot B_{\text I} \,+\,  A_{\text{II}} \cdot B_{\text{II}}
\,+\,\cdots \,+\, A_{\text{VII}} \cdot B_{\text{VII}},
\end{equation}
where $A_{\text I} = 105$ is the number of configurations of type I, $B_{\text I}$ is the multiplicity of each
$A_{\text I}$ ray in ${\rm trop}(C)$, and likewise for the other types. 
In other words, $B_{\text I}$ is the number of Puiseux series solutions to  \eqref{eq:critpt48t} 
whose $t$-adic valuation is the representative of type I in Table \ref{tab:combtypes48}. 

We use numerical computation to obtain the multiplicities $B_{\text I}, \ldots, B_{\text{VII}}$.
This is done without any prior knowledge about ${\rm trop}(X(4,8))$ or the fibration $\pi_{4,8}$.
 Approximate solutions of \eqref{eq:critpt48} 
 are tracked numerically along the soft limit degeneration to  
 learn the tropical critical points \eqref{eq:valpatterns}. 
 We explain the details of this computation in a more general context in Section~\ref{section8}.
 
  We filter the obtained list of candidate critical points by recording the regression error in \eqref{eq:regerr} and by checking that the solutions come in $S_7$-orbits. This gives a total of 
  $3\,151$ successfully found vectors $\hat{q}$ in $\QQ^{70}$. One of them is
 $0 \in \RR^{70}$. This establishes Theorem*~\ref{conj:48}.
 To learn the multiplicities $B_\bullet$ in (\ref{eq:decompchi48}),
we record, for each tropical solution $\hat{q}$ of type  $\bullet$,
   the number of classical solutions that were found to have valuation $\hat{q}$. 
    The multiplicities are 
\begin{equation}
\label{eq:48Multiplicities} 
\!\! \begin{matrix} 
B_{\text I} = 5\,680, \! & B_{\text{II}} = 1\,704,  \! & B_{\text{III}} = 988, \! & B_{\text{IV}} = 832, \! &
B_{\text V} = 308,\! & B_{\text{VI}} = 240,\! & B_{\text{VII}} = 72. 
\end{matrix}
\end{equation}
Plugging these values into \eqref{eq:decompchi48} leads to Theorem* \ref{conj:chi48}. 
Additional strong support arises from the fact that
the number $5\,211\,816$ is also the number of approximate solutions found
in a stand-alone run of \texttt{monodromy\_solve}, although not all of them could be certified via
\cite{BRT}.

\begin{remark}
The decomposition \eqref{eq:decompchi48} is strongly related to the sum in
Lemma \ref{lemma:stratifiedfibration}. For instance,
the formula $\chi(X(3,7)) = 26 \cdot 42 + 15 \cdot 12$ from the proof of Theorem \ref{thm:numbersk=3}
partitions the $1\,272$ critical points of the log-likelihood function into 
$1\,092$ solutions that converge in $X(3,7)$ and 180 solutions that move to the boundary 
in the soft limit. These boundary solutions escape from 
the open variety $X(3,7)$ in $15$ groups of $12$. 
Here $A_{\text I} =  15$ is a combinatorial number, and $B_{\text I} = 12$ is the ML degree of a stratum in $X(3,6)$. 
This is isomorphic to the codimension 3 matroid stratum in $X(3,7)$ with ML degree 12 in Theorem \ref{thm:matroid_ml}. A similar interpretation holds for $X(3,8)$ and $X(3,9)$. In the case of $X(4,8)$, 
however, not all ray multiplicities $B_\bullet$ are equal to the ML degree of a corresponding stratum in $X(4,7)$. A notable difference to the $k = 3$ case is that the matroid stratum 
of type I, seen for $k=4,m=8$
in Table \ref{tab:combtypes48}, differs in dimension from its corresponding stratum in $X(4,7)$.
\end{remark}

\section{Statistical models and their tropicalization}
\label{section7}

We consider maximum likelihood estimation (MLE) for discrete statistical models \cite{HS, PS}.
Our  conventions and notation will be as in \cite{ST}.
The given model is a
$d$-dimensional subvariety $X$ of the projective space $\PP^n$, which is
assumed to intersect the simplex $\Delta_n \subset \RR\PP^n$ of positive points.
We seek to compute the critical points on $X$ of the log-likelihood function
$$
{\cal L}_u \,=\,
u_0 \cdot {\rm log}(p_0) \,+ \,
u_1 \cdot {\rm log}(p_1) \,+ \,\cdots + \,u_n \cdot {\rm log}(p_n) \,- \,
(u_0{+}u_1{+}\cdots{+}u_n) \cdot {\rm log}(p_0{+}p_1{+} \cdots {+} p_n).
$$
For any positive constants $u_0,u_1,\dots,u_n$, representing  data
in statistics, this is a well-defined function on $\Delta_n$.
The aim of likelihood inference  is to maximize ${\cal L}_u$
over all points $p$ in the model $X \cap \Delta_n$.
The number of all complex critical points,
for generic $u$, is the ML degree
of the model $X$.  If $X$ is smooth then
this equals the signed Euler characteristic of the open variety
$X^\circ$, which is the complement of the
divisor in $X$ defined by  $p_0 p_1 \cdots p_n (\sum_{i=0}^n p_i) = 0$.

Likelihood degenerations  were first introduced
in the setting of algebraic statistics by Gross and Rodriguez in \cite{GR},
who studied the behavior of the MLE when some of the $u_i$ approach zero. 
They distinguish between  model zeros, structural zeros and 
sampling zeros.  These statistical concepts can serve as a guide
for interpreting likelihood degenerations.

We draw samples independently from some unknown
distribution that is in $X \cap \Delta_n$. The probabilities in the
following definitions refer to that sampling distribution.
The data are summarized in a vector $u \in \NN^{n+1}$,
where $u_i$ denotes the number of observations found to be in state $i$.
Suppose that state $i$ was never observed in our sample. The entry
$u_i = 0$ is called:
\begin{itemize}
\item a {\it structural zero} if the probability of it being zero is equal to one; \vspace{-0.2cm}
\item a {\it sampling zero} if the probability of it being zero is less than  one; \vspace{-0.2cm}
\item a {\it model zero} if the maximizer of ${\cal L}_u$ over $X \cap \Delta_n$ 
is a critical point of  the restriction of ${\cal L}_u$ to the hyperplane section $X \cap \{p_i=0\}$.
\end{itemize}

Structural zeros may mean that the wrong model was chosen,
so we exclude this possibility. What remains is a consideration of
sampling zeros and model zeros.
These statistical concepts led Huh and Sturmfels to propose the following formula 
in \cite[Conjecture~3.19]{HS}:
\begin{equation}
\label{eq:huh}
 {\rm MLdegree}(X) \,\, = \,\,
 {\rm MLdegree}(X \cap \{ p_j = 0 \} )\, + \,
 {\rm MLdegree}(X|_{u_j = 0}). 
 \end{equation}
 The last summand counts  critical points of ${\cal L}_u$ on 
 $X$ when $u_j = 0$ and~the other  $u_i$ are~generic.
 The identity (\ref{eq:huh}) holds under suitable
 smoothness and transversality assumptions.
 They ensure that the ML degrees are signed
 Euler characteristics of very affine varieties,
  obtained by removing the arrangement
 $\mathcal{H} = \{ p_0 p_1 \cdots p_n (\sum_{i=0}^n p_i) = 0\}$
 of $n+2$ hyperplanes  from $\PP^n$.
For the last summand we remove only $n+1$
hyperplanes. The Euler characteristic is additive
relative to the additional hyperplane $\{p_j = 0\}$.
The sum becomes a minus for the signed Euler characteristic,
as the dimensions differ by one, so
 the identity (\ref{eq:huh}) follows.

Familiar combinatorics arises when $X$ is a linear space. In this case,
 $X^\circ$ is the complement of
 $n+1$ hyperplanes in affine $d$-space. 
The number of bounded regions can be computed by
deletion-restriction. This is precisely the formula in~(\ref{eq:huh}).
Moreover, all critical points are real, and there is one critical point per bounded region.
An example from \cite{ST} is the space $X = X(2,m)$
of $m$ points on the line $\PP^1$, modulo projective transformations.
Here $d = m-3$, $n = m(m-3)/2$, and the ML degree equals $(m-3)!$.
The formula (\ref{eq:huh}) is essentially that for
soft limits in \cite{CEGM, CUZ}. 
We count solutions in (\ref{eq:huh}) as
singular solutions plus regular solutions.

In this section, we introduce a vast generalization of
soft limits, namely {\em tropical degenerations}.
We examine MLE for discrete statistical models
through the lens of tropical geometry. In what follows,
the real numbers $\RR$ are replaced by the real Puiseux series $R = \RR\{\! \{t \} \! \}$.
This is a real closed field, and it comes with the $t$-adic valuation.
The uniformizer $t$ is positive and infinitesimal. A scalar $u$ in $R$
can be viewed as the germ of a function $u(t)$ near $t \rightarrow 0$.

We are now given $u_0,\ldots,u_n \in R$, with
valuations $w_i = {\rm val}_t(u_i)$. We call
 $w = (w_0,\ldots,w_n)$  the {\em tropical data vector}.
Each  critical point $\hat p$ of ${\cal L}_u$ has its coordinates $\hat p_i$ in 
the algebraic closure  $K=\CC\{\!\{ t \}\!\}$ of the ordered field $R$.
We set $\hat q_i = {\rm val}_t(\hat p_i)$, and we refer to
$\hat q = (\hat q_0,\ldots,\hat q_n)$ as a {\em tropical critical point}.
Given any model $X$, we would like to describe the multivalued map
that takes a tropical data vector $w$ to the set of its tropical critical points~$\hat q$.

The following theorem accomplishes this goal for the class of
linear models \cite[\S 1.2]{PS}.  We augment the homogeneous linear forms
defining $X$  by the equation $p_0 + p_1 + \cdots + p_n = 1$,
and we identify $X$ with the resulting $d$-dimensional
affine-linear subspace in $\RR^{n+1}$.
We write $X^\perp$ for the linear subspace of
$\RR^{n+1}$ that consists of all vectors perpendicular to $X$,
with respect to the usual dot product. Thus $X^\perp$ is a 
vector space of dimension $n-d+1$ in~$\RR^{n+1}$.

The tropical affine space ${\rm trop}(X)$ 
is a pointed cone of dimension $d$  in $\RR^{n+1}$. Combinatorially,
this is the Bergman fan  \cite[\S 4.2]{MS}
of the rank $d+1$ matroid on $n+2$ elements
defined by $X$. Here the matroid is associated with the hyperplane
arrangement $X \backslash X^\circ$. The tropical linear
space ${\rm trop}(X^\perp)$ has dimension $n-d+1$.
It is a fan with $1$-dimensional lineality space spanned by $(1,1,\ldots,1)$.
Combinatorially, it is the Bergman fan
of the rank $n-d+1$ matroid on $n+1$ elements defined by $X^\perp$.
Here  the matroid of $X^\perp$ is the dual of a one-element contraction of the matroid of $X$.
The contracted element corresponds to the hyperplane at infinity, namely
$\{p_0+p_1 + \cdots + p_n = 0\}$. It is very important to distinguish this element.

\begin{theorem} \label{thm:tropMLE}
	If the tropical data vector $w$ is sufficiently generic then 
	there are exactly ${\rm MLdegree}(X)$ many distinct tropical critical points. They are given by the intersection
	\begin{equation}
	\label{eq:troptrop}
	\hat q \,\,\in \,\, {\rm trop}(X) \,\, \cap \,\, (w - {\rm trop}(X^\perp)). 
	\end{equation}
\end{theorem}

We call the subspace $X$  {\em general} if
both of the above matroids  are uniform.
In that case we have
${\rm MLdegree}(X) = \binom{n}{d}$; see \cite[Example 4]{ST}.
The matroid of $X$ is the uniform matroid $U_{d+1,n+2}$,
and the matroid of $X^\perp$ is the uniform matroid $U_{n-d+1,n+1}$.
We abbreviate $e_{n+1} = - e_0 - e_1 - \cdots - e_n$, the negated
sum of all unit vectors in $\RR^{n+1}$.
The tropical affine space ${\rm trop}(X)$ is the union of all
$\binom{n+2}{d}$ cones $ {\rm pos}(e_i: i \in I)$ where $I$
runs over $d$-element subsets of $\{0,1,\ldots,n,n+1\}$.
The tropical linear space ${\rm trop}(X^\perp)$ is the union of
all $\binom{n+1}{n-d}$ cones $\RR e_{n+1} + {\rm pos}(e_j: j \in J)$,
where $J$ runs over  $(n-d)$-element subsets of $\{0,1,\ldots,n\}$.~Let $w \in \RR^{n+1}$ and suppose 
$w_0$ is its smallest coordinate. Then
$w+ w_0 e_{n+1}$ is a nonnegative~vector with first coordinate $0$.
Replacing $w$ by $w+w_0 e_{n+1}$ does not change tropical critical points.

\begin{corollary} \label{cor:uniform}
	If $X$ is general then (\ref{eq:troptrop}) consists of the
	$\binom{n}{d}$ points  $\, \hat q \, =  \sum_{i \in I} (w_i - w_0) \, e_i\,$,
	where $I$ runs over all $d$-element subsets of $\{1,2,\ldots,n\}$.
	These are the tropical critical points.
\end{corollary}

\begin{proof}
	By construction, the vector $\hat q$ lies in ${\rm trop}(X)$.
	We also have
	$$ w-\hat q \,\,=\,\, w - \sum_{i \in I} w_i e_i + \sum_{i \in I} w_0 e_i \, \,=\,\,
	 \sum_{j \not\in I} (w_j - w_0) e_j  - w_0 e_{n+1}.
	$$
	This vector is visibly in ${\rm trop}(X^\perp)$.
	We have thus constructed $\binom{n}{d}$ distinct vectors in the intersection (\ref{eq:troptrop}).
	The corollary hence follows from Theorem \ref{thm:tropMLE}, applied to general $X$.
\end{proof}

\begin{example}[$n=3,d=2$]
	We consider the $2$-parameter linear model $X$ for
	the state space $\{{\tt A}, {\tt C}, {\tt G}, {\tt T}\}$ discussed in
	\cite[Example 1.1]{PS}.  This model is defined by one homogeneous
	linear constraint $\, c_{\tt A}  p_{\tt A} + 
	c_{\tt C}  p_{\tt C} + c_{\tt G}  p_{\tt G} + c_{\tt T}  p_{\tt T}  = 0$.
	Its coefficients
	$c_i$ are nonzero real numbers. Consider the data vector
	$u = (t^{w_{\tt A}},t^{w_{\tt C}},t^{w_{\tt G}},t^{w_{\tt T}})$.
	The four exponents are nonnegative integers, and we 
	assume that $w_{\tt A}$ is the smallest among them.
	The log-likelihood function
	has three critical points $\hat p = (\hat p_{\tt A},  \hat p_{\tt C}, \hat p_{\tt G}, \hat p_{\tt T})$,
	one for each bounded polygon.
	These are functions in $t$,
	and we seek their behavior for $t \rightarrow 0$.
	This is given by the three tropical critical points:
	$$ \hat q \,\, = \,\, 
	(0, w_{\tt C}-w_{\tt A},w_{\tt G}-w_{\tt A},0), \,\,
	(0, w_{\tt C}-w_{\tt A}, 0 ,w_{\tt T}-w_{\tt A}) \,\,\,{\rm or} \,\,\,
	(0, 0, w_{\tt G}-w_{\tt A},w_{\tt T}-w_{\tt A}).
	$$
	We change the model by setting $c_{\tt A} = 0$, so $X$ is no longer general.
	The number~of bounded regions drops from $3$ to $2$.  The matroid of $X$ changes,
	and so does ${\rm trop}(X)$. We now find
	$
	\hat q \, = \, (0, w_{\tt C}-w_{\tt A},w_{\tt C}-w_{\tt A},w_{\tt G}-w_{\tt A}) \, {\rm or} \,
	(0, w_{\tt C}-w_{\tt A},w_{\tt G}-w_{\tt A},w_{\tt C}-w_{\tt A})$
	provided $w_{\tt C} < {\rm min}(w_{\tt G},\! w_{\tt T})$.
\end{example}

\begin{proof}[Proof of Theorem \ref{thm:tropMLE}]
	We use the formulation of the likelihood correspondence given in
	\cite[Proposition 1.19]{HS}.
	This states, in our notation, that the critical points $\hat p$ are 
	the elements of
	\begin{equation}
	\label{eq:intersection}
	X \,\,\cap \,\,(u^{-1} \star X^\perp)^{-1} .
	\end{equation}
	Here, $\star $ is the Hadamard product, and $u^{-1}$ is the
	coordinatewise reciprocal of the vector $u$. The intersection
	in (\ref{eq:intersection}) commutes with tropicalization,
	provided $w = {\rm val}_t(u)$ is generic:
	\begin{equation}
	\label{eq:stableintersection}
	{\rm trop}(X \,\cap \,(u^{-1} \star X^\perp)^{-1} ) \,\,= \,\,
	{\rm trop}(X) \,\cap \,\, - {\rm trop}(u^{-1} \star X^\perp) \,\, = \,\,
	{\rm trop}(X) \,\cap \, ( w - {\rm trop}(X^\perp)). 
	\end{equation}
	Indeed, the left expression is contained in the middle expression,
	and they are equal in the sense of stable intersection.
	This is the content of \cite[Theorem 3.6.1]{MS}. 
	 In (\ref{eq:stableintersection}) we intersect
	polyhedral spaces of dimensions $d$ and $n+1-d$ in $\RR^{n+1}$.
	The second equation follows from \cite[Proposition 5.5.11]{MS}.
	Since  $w$ is generic, the intersection
	is transverse at any intersection point and each intersection point is isolated.
	Lemma~\ref{lem:bergman} below shows that the multiplicity of every tropical intersection is $1$, even in the more general case of a nonrealizable matroid.
\end{proof}

	Fix a matroid $X$ of rank $d+1$ on the elements $\{0,1,\ldots, n+1\}$.
	Let $X^\perp = (X/(n+1))^\ast$ be  the dual of the contraction of $X$ by the element~$n+1$.
	Furthermore, let $\mathcal{F}\coloneqq \{F_1 \subsetneq \dots \subsetneq F_d\}$ and $\mathcal{F}^\perp\coloneqq \{F_1^\perp \subsetneq \dots \subsetneq F_{n-d}^\perp\}$ be flags of flats of $X$ and $X^\perp$, respectively with ${\rm rank}(F_i)=i$ and ${\rm rank}(F_j^\perp)=j$ for all $i,j$.
	Since X has rank $d+1$, we can assume $F_d\subseteq \{0, \ldots, n\}$.

\begin{lemma}\label{lem:bergman}
Each intersection in \eqref{eq:stableintersection} has multiplicity $1$. It is the
 signed determinant of an $(n+1)\times (n+1)$ matrix
		whose columns are indicator vectors of flats in flags as above:
	\[
	M_{\mathcal{F},\mathcal{F}^\perp} \,\,\coloneqq\,\,
	\begin{pmatrix}
	\vline  & \dots & \vline & \vline  & \dots & \vline & \vline \\
	\, e_{F_1}  & \dots & e_{F_d} & e_{F_1^\perp}  & \dots & e_{F_{n-d}^\perp} & e_{\{0,\ldots,n\}}\\
	\vline   & \dots & \vline & \vline   & \dots & \vline & \vline 
	\end{pmatrix}.
	\]
Such a matrix $M_{\mathcal{F},\mathcal{F}^\perp}$ has determinant $0$ or~$\pm 1$.
Moreover, if $M_{\mathcal{F},\mathcal{F}^\perp}$ 
  is invertible, there exist complementary bases of the matroids $X/(n+1)$ and $X^\perp$ generating the flags $\mathcal F$ and $\mathcal{F}^\perp$.
\end{lemma}

\begin{proof}
We proceed along the following four steps.
\begin{enumerate}
\item We can assume that $\{0,\dots,d-1,n+1\}$ is a basis of $X$, and
		\begin{equation}\label{eq:MF}
		M_{\mathcal{F},\mathcal{F}^\perp} \,\,= \,\,
\begin{small}		\begin{pmatrix}
		\begin{matrix}
		1 & 1 & 1 & 1\\
		0 & 1 & 1 & 1\\
		0 & 0  & \ddots & 1\\
		0 & 0  & 0 & 1
		\end{matrix}
		& \rvline & & \ast & \\
		\hline
		& \rvline \\
		\ast & \rvline & &
		\ast \\
		&\rvline  &
		\end{pmatrix}, \end{small}
		\end{equation}
		where the left block has $d$ columns and the symbols `$\ast$' represent arbitrary $\{0,1\}$-entries.
The definition of flag ensures that within each column block, a $1$-entry is followed by only $1$'s in the same row. Therefore, rows in each column block are uniquely determined by their number of 1-entries.
\item Let $I_j$ be the set of row indices corresponding to rows having precisely $j$ $1$-entries in the left block. For each row index $i$, let $\gamma(i)$ be the smallest column index in the right block such that the entry $(i, \gamma(i))$ is $1$. By permuting rows of~$M_{\mathcal{F},\mathcal{F}^\perp}$ we can ensure that 
all rows indexed by $I_j$ are sorted by increasing number of~$1$'s in the right block. Since a $1$-entry in the right block is followed by~$1$'s in the same row, we have that
\begin{equation} \label{eq:gamma}
\text{if $i, i' \in I_j$ are such that } i' < i,~ \text{then } \gamma( i ) \leq \gamma( i' ).
\end{equation}
If equality holds for some $i \in I_j \setminus \{ i' \}$, then two rows are equal and $M_{\mathcal{F},\mathcal{F}^\perp}$ has determinant~$0$, in which case we are done. Hence, in what follows we assume that the last inequality in \eqref{eq:gamma} is strict for all distinct $i, i' \in I_j$. 
\item We perform the following elementary row operations. Each row $r_i$ indexed by $i > d$ is replaced by $r_i - r_{i'}$, where $i \in I_{i'}$. After this operation, the lower left block in $M_{\mathcal{F},\mathcal{F}^\perp}$ is zero, all entries are still 0 or 1 and the function $\gamma$ is unchanged.
\item It remains to show that the determinant of the lower right block matrix is $\pm 1$. Since the matrix $M_{\mathcal{F},\mathcal{F}^\perp}$ is still of the form~\eqref{eq:MF}, the elements $\{0,\dots,d-1,n+1\}$ are still a basis of $X$ (after the above permutations).
		Therefore by definition of $X^\perp$, the set $\{d,\dots,n\}$ is a basis of $X^\perp$.
		This implies that, up to a permutation of rows, the lower right block is an upper triagonal matrix with $1$-entries on the diagonal.
\qedhere
	\end{enumerate}
\end{proof}

We now apply Theorem \ref{thm:tropMLE} to the CHY model
 $X^\circ =  X(2,m)$, where $n = m(m-3)/2$. The tropical linear space
${\rm trop}(X)$ consists of ultrametrics on $m-1$ points \cite[Lemma 4.3.9]{MS}.
The matroid of $X$ is the graphic matroid of the complete graph $K_{m-1}$; see
\cite[Example 4.2.14]{MS}. 
Using (\ref{eq:generalmatrix}) with $k=2$,
 the  vertices of $K_{m-1}$ are labeled by $2,3,\ldots,m$.
The special edge $e=\{2,3\}$ corresponds to the hyperplane at infinity.
The matroid of $X^\perp$ is the cographic matroid of
the graph $K_{m-1}/e$ obtained by contracting $e$.
This is dual to the graphic matroid of $K_{m-1}/e$.
Vectors in ${\rm trop}(X)$ have the minimum 
attained twice on every circuit of $K_{m-1}$, where the edge $e$
has weight $0$. Vectors in ${\rm trop}(X^\perp)$ attain their minimum
twice on every cocircuit of $K_{m-1}/e$.
Thus tropical MLE amounts to writing
the vector $w$ as a sum of
two such minimum-attained-twice vectors.
The number of such decompositions equals $(m-3)!$.

\begin{example}[$m=6$]
	Following \cite[Example 2]{ST}, we consider the CHY model $X(2,6)$.
	This corresponds to an arrangement of $9$ planes in $\RR^3$, with
	six bounded regions, namely the tetrahedron of a  triangulated $3$-cube.
	We coordinatize this model by the matrix
(\ref{eq:generalmatrix}),	
		so that $e = \{2,3\}$ is the special edge of $K_5$.
	Our data are the	Mandelstam invariants 
	$$ u_{24} = t^{12},\, u_{25} = t^6,\, u_{26} = t^9, \,u_{34} = t^{12}, \,u_{35} = t^5, \,
	u_{36} = t, \, u_{45} = t^{10},\,
	u_{46} = t^{11}, \,u_{56} = t^3 .$$
	Hence the tropical data vector equals
	$ \,w \,= \, (w_{24}, w_{25}, \ldots, w_{45}) \, = \,
	(12, 6, 9, 12, 5, 1, 10, 11, 3) $.
	Our task is to find all additive decompositions
	$\,w \,=\, \hat q \,+\, (w-\hat q)$ where the two summands lie
		in the respective Bergman fans ${\rm trop}(X)$ and ${\rm trop}(X^\perp)$.
		One such decomposition equals
\begin{equation}
\label{eq:decomp}
w \,\,= \,\, \hat q \, + \, (w-\hat q) \,\, = \,\,
( 7,  5, 2,   0, 0, 0,   5,  2,  2) \,+ \,
( 5,  1, 7,  12, 5,1,   5, 9, 1).
\end{equation}
This solution is verified in Figure \ref{fig:planarduals}:
the minimum over each  circuit is attained at least twice.

\begin{figure}[h]
		\vspace{-0.09cm}
		\begin{center}
			\includegraphics[scale=0.5]{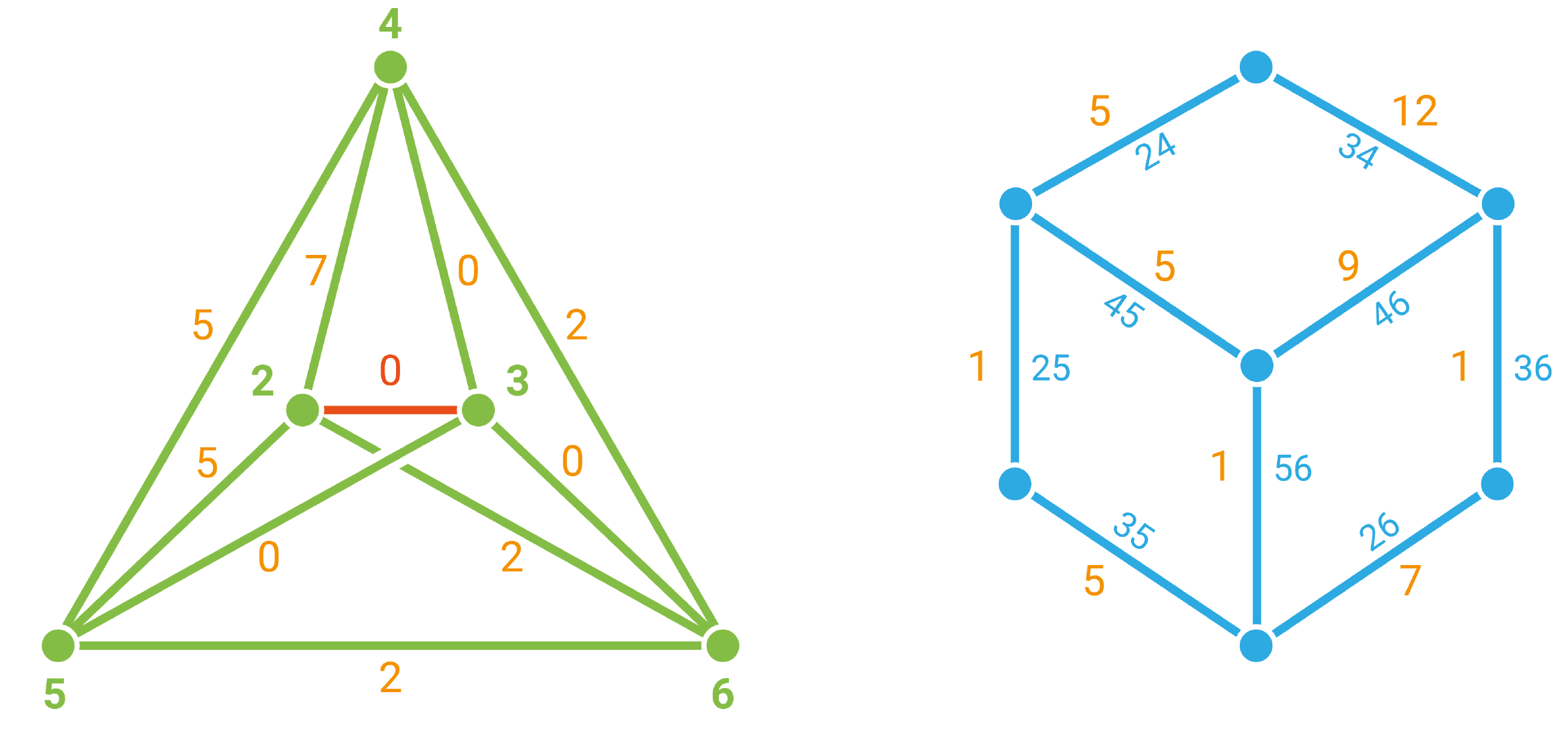} 
		\end{center}
		\vspace{-0.68cm}
		\caption{Contracting the edge $23$ in the green graph $K_5$ yields a planar graph
			whose dual is the blue graph on the right. The orange edge weights are the
			entries in (\ref{eq:decomp}).
			The key  property is that
			the minimum is attained at least twice in each cycle of the respective graph.
			\label{fig:planarduals}}
	\end{figure}

We find that	
	 the intersection 
	(\ref{eq:troptrop}) consists of six points. The tropical critical points $\hat q$  are
	$$ \begin{matrix}
	(0, 0, 8, 4, 2, 0, 2, 0, 0) \,,& &
	(0, 5, 2, 2, 0, 0, 0, 0, 2) \,,& &
	(1, 0, 8, 0, 2, 0, 0, 1, 0)\,, \\
	(2, 5, 2, 0, 0, 0, 2, 3, 2)\,, & &
	(7, 5, 2, 0, 0, 0, 5, 2, 2)\,, & &
	(9, 0, 8, 0, 2, 0, 0, 8, 0) .
	\end{matrix} $$
Each $\hat q$ gives a decomposition as in (\ref{eq:decomp}),
	where the two summands are  compatible with the cycles of the two graphs in
	Figure~\ref{fig:planarduals}.
		These solutions specify small arcs
	that lie in the six tetrahedra.
	These arcs converge  for $t \rightarrow 0$ with the given rates
		to  vertices  of the arrangement. 

What we have outlined here is	 the combinatorial theory of
 {\em tropical CHY scattering}. 	This works for all
 $m \geq 5$.	The soft limits of \cite{CUZ} arise as the
	very special case when
	the tropical Mandelstam invariants $w$  satisfy
	$w_{im} = 1$ and $w_{ij} = 0$ for $i \leq j \leq m-1$.
The key player in tropical CHY scattering
 is the space  ${\rm trop}(X)$ of ultrametric phylogenetic trees. In the case
	 $m=5$, this space is a cone
		over the Petersen graph.	
See \cite[\S 4.3]{MS} for details.
\end{example}

In this section we presented the theory of tropical MLE for linear models.
This raises the question of what happens when the model $X$ is 
 an arbitrary (nonlinear) projective variety in $\PP^n$. A partial answer based on homotopy techniques is given in the next section. 
The role of (\ref{eq:intersection}) is now played by the likelihood
correspondence. According to \cite[Theorem 1.6]{HS},
this is an $n$-dimensional subvariety in $\PP^n \times \PP^n$.
An ambitious goal is to determine the tropicalization of that subvariety.
The desired pairs $(w,\hat q)$ are points in
that {\em tropical likelihood correspondence}. 
This leads to very interesting geometry and combinatorics. For instance, it is closely related to
the Bernstein-Sato slopes studied by van der Veer and Sattelberger~\cite{VS}.

\section{Learning valuations numerically} \label{section8}

Previous sections developed two techniques for obtaining the ML degree of a very affine variety: 
Euler characteristics via stratified fibrations, and tropical geometry. Each method leads to a meaningful combinatorial description of the ML degree, 
but it requires significant combinatorial efforts. In this section we propose a numerical method which computes the tropical solutions $\hat{q}$ discussed in 
Sections \ref{section6} and \ref{section7} directly, while avoiding any combinatorial overhead. A decomposition of the ML degree as a positive sum of integers naturally emerges as a byproduct. We used this method to verify the multiplicities \eqref{eq:48Multiplicities} as detailed in Section~\ref{section6}.

Let $K = \CC \{\!\{ t\}\!\}$ and consider a very affine variety $X \subset (K^*)^n$ 
that is defined over $\CC$. Fix a data vector $u \in K^n$.
The problem of tropical MLE is to compute the coordinate-wise valuations 
$\hat{q} = {\rm val}_t(\hat{p})$ of the critical points $\hat{p}$ 
of the log-likelihood function ${\cal L}_u$ restricted to~$X$.
Note that each $\hat{q}$ is a vector in $\QQ^n$, so the output can be written in exact arithmetic.

In this section, we think of $t$ as a coordinate for $\CC^*$.
We assume for simplicity that $u \in \mathbb{C}[t,t^{-1}]^n$ is given by Laurent polynomials.
The following incidence variety is a curve:
\[
C \,\,=\,\, \bigl\{ (\hat{p}, t) \in X \times \CC^* ~|~ \hat{p} \text{ is a critical point of } ({\cal L}_{u(t)}){\big|}_X \bigr\}.
\]
 The generalization to the case where the $u_i$ are Laurent series, convergent in the punctured disk $\{ t \in \CC^* ~|~ |t| < 1 \}$, is straightforward. 
We assume that $u$ is sufficiently generic, so that:
\begin{itemize}
\item the projection map $\pi_{\CC^*}: C \to \CC^*$ is an MLdegree$(X)$-to-one branched covering,
\item the half-open real line segment $(0,1] \subset \CC^*$ avoids the branch locus of $\pi_{\CC^*}$. 
\end{itemize}

Any point $(\hat{p}(1),1)$ of the curve $ C$ lies on a unique path in 
$\pi_{\CC^*}^{-1}((0,1])$ which corresponds to
 a critical point $\hat{p}(t)=(\hat{p}_1(t),\ldots,\hat{p}_n(t)) \in K^n$. Each
coordinate of $\hat{p}(t)$ is a Puiseux series
\begin{equation} \label{eq:approxclass}
\hat{p}_i(t) = c_i t^{\hat{q}_i} + \text{ higher order terms},  
\end{equation}
where $\hat{q} = (\hat{q}_1,\ldots,\hat{q}_n)$ is its corresponding tropical critical point. 

For any real constant $t' \in (0,1]$ with $t'\ll 1$, the following approximation holds:
\begin{equation} \label{eq:approxval}
\log | \,\hat{p}_i(t') \,| \approx \,\,\log |c_i| + \hat{q}_i \cdot \log t' . 
\end{equation}
Hence, for all values of $t'$ that are small enough, 
the points $(a_{t'},b_{t'}) = (\log t',\log |\,\hat{p}_i(t')\,| )\in \RR^2$ lie approximately on a line with slope $\hat{q}_i = {\rm val}_t(\hat{p}_i(t))$. We wish to learn that slope.

Using standard numerical predictor-corrector techniques on the critical point equations, one can compute $(\hat{p}(t'),t') \in \pi^{-1}_{\CC^*}(t')$ for any $t' \in (0,1]$. This amounts to evaluating the solution $\hat{p}(t)$ at $t'$. 
 In our project, we used {\tt HomotopyContinuation.jl} \cite{HCjl} for this computation.

By evaluating $\hat{p}(t)$ at many $t'$ near $0$, we may approximate 
the $i$th coordinate $\hat q_i$ by fitting a line through the points $(a_{t'},b_{t'})$.
See (\ref{eq:approxval}). We find $\hat{q}$ by doing this for $i=1,2,\ldots,n$.
This discussion is summarized  in Algorithm \ref{algo:computetropicalcriticalpoints} for computing the tropical critical point $\hat{q}$. 

\begin{algorithm}[!htpb]
\SetKwIF{If}{ElseIf}{Else}{if}{then}{elif}{else}{}%
\DontPrintSemicolon
\SetKwProg{Computation of tropical critical points}{Computation of tropical critical points}{}{}
\LinesNotNumbered \smallskip
\KwIn{Data $u$ and a critical point $\hat{p}(1) \in (\CC^*)^n$ of the log-likelihood function~${\cal L}_{u(1)}$.}
\KwOut{The tropical critical point $\hat{q}$ associated to the solution $\hat{p}(t)$,
together with some measure $\rho$ of trust in this result.} \smallskip
\nl  With $ \hat{p}(1)$ as the start solution, use numerical predictor-corrector techniques to compute $\hat{p}(t')$ for a finite set $T \subset (0,1]$ of values $t'$, all sufficiently close to $0$. \; 
\nl Fit lines $\ell_i$ through the data points $\{(\log t, \log |\hat{p}_i(t)|) ~|~ t \in T \}$ for each coordinate $\hat{p}_i$, using standard regression techniques. \;
\nl Using any rationalization heuristic, such as the built-in function $\texttt{rationalize}$ in 
$\texttt{Julia}$, round the slope of the line $\ell_i$ to obtain the rational number $\hat{q}_i$. \;
\nl As a measure of trust, record the \emph{regression error} 
\begin{equation} \label{eq:regerr}
\rho \,\,:= \,\,\sum_{t' \in T} \bigl( \log |\hat{p}_i(t')| - ( \log |c_i| + \hat{q}_i \cdot \log t')  \bigr)^2.
\end{equation} \;
\vspace{-0.8cm}
\nl \Return{} $(\hat{q},\rho)$
\caption{Numerical computation of tropical critical points \label{algo:computetropicalcriticalpoints}}
\end{algorithm} 


\begin{example}[Soft limits]
\label{ex:trop37}
We apply this to the CEGM model $X(3,7)$. It is  parametrized by the 
$3 \times 3 $ minors $p_{ijk}$ of the matrix $M_{3,7}$ given in \eqref{eq:generalmatrix}. 
Consider the log-likelihood function 
\begin{equation}
\label{eq:loglike37}
{\cal L}_{u(t)} \,\,=\,\, \tilde{\cal L}_{3,7}(x,t) \,\,\,=
\sum_{1 \leq i < j < k \leq 6} \!\!\!\!  u_{ijk} \cdot \log p_{ijk}\,\,\, +
 \sum_{1 \leq i< j \leq 6} \!\!\!\! t \cdot u_{ij7} \cdot \log p_{ij7} 
 \end{equation}
for generic complex parameters $u_{ijk}$.
The numerical computation in \cite[\S 4]{ST} provides $ 1\,272$ start solutions $\hat{p}(1)$.
Each of these now serves as the input to Algorithm \ref{algo:computetropicalcriticalpoints}.
That algorithm performs the likelihood degeneration
(soft limit) purely numerically for $t \rightarrow 0$.

The $1\,272$ runs of Algorithm \ref{algo:computetropicalcriticalpoints} lead to only
$16$ distinct tropical critical points $\hat{q}$.
This means that different start solutions $\hat{p}(1)$ yield the same output.
The multiplicity of $\hat{q}$ in the tropical curve ${\rm trop}(C)$ is the
number of runs that yield output $\hat{q}$.
The coordinates and multiplicities of all tropical critical points
are shown in the columns of Table \ref{tab:37}. 
We stress that this table was computed blindly, without any prior knowledge
about the model $X(3,7)$.

Remarkably, one learns the geometry of the stratified fibration $\pi_{3,7}$  from the
output in Table~\ref{tab:37}.
The first column corresponds to the $1\,092$ regular solutions, i.e., those which converge in $X(3,7)$. 
The others correspond to $15$ groups of $12$ boundary solutions, whose limit for $t \rightarrow 0$ lies on the boundary in the tropical compactification of $X(3,7)$.
Thus, by using Algorithm \ref{algo:computetropicalcriticalpoints},
we  discover the ML degrees in Theorem \ref{thm:numbersk=3} 
in a purely automatic manner.

\begin{table}[!htpb]
\centering
\tiny
\setlength{\tabcolsep}{0.25cm}
\begin{tabular}{c|cccccccccccccccc}
$ p_{123} $ & 0 & 0 & 0 & 0 & 0 & 0 & 0 & 0 & 0 & 0 & 0 & 0 & 0 & 0 & 0 & 0\\ 
$ p_{124} $ & 0 & 0 & 0 & 0 & 0 & 0 & 0 & 0 & 0 & 0 & 0 & 0 & 0 & 0 & 0 & 0\\ 
$ p_{125} $ & 0 & 0 & 0 & 0 & 0 & 0 & 0 & 0 & 0 & 0 & 0 & 0 & 0 & 0 & 0 & 0\\ 
$ p_{126} $ & 0 & 0 & 0 & 0 & 0 & 0 & 0 & 0 & 0 & 0 & 0 & 0 & 0 & 0 & 0 & 0\\ 
$ p_{127} $ & 0 & 0 & 0 & 0 & 0 & 0 & 0 & 0 & 0 & 0 & 0 & 0 & 0 & 0 & 0 & 0\\ 
$ p_{134} $ & 0 & 0 & 0 & 0 & 0 & 0 & 0 & 0 & 0 & 0 & 0 & 0 & 0 & 0 & 0 & 0\\ 
$ p_{135} $ & 0 & 0 & 0 & 0 & 0 & 0 & 0 & 0 & 0 & 0 & 0 & 0 & 0 & 0 & 0 & 0\\ 
$ p_{136} $ & 0 & 0 & 0 & 0 & 0 & 0 & 0 & 0 & 0 & 0 & 0 & 0 & 0 & 0 & 0 & 0\\ 
$ p_{137} $ & 0 & 1 & 0 & 0 & -1 & 0 & 1 & 0 & 1 & 0 & 0 & 0 & 0 & -1 & 0 & -1\\ 
$ p_{145} $ & 0 & 0 & 0 & 0 & 0 & 0 & 0 & 0 & 0 & 0 & 0 & 0 & 0 & 0 & 0 & 0\\ 
$ p_{146} $ & 0 & 0 & 0 & 0 & 0 & 0 & 0 & 0 & 0 & 0 & 0 & 0 & 0 & 0 & 0 & 0\\ 
$ p_{147} $ & 0 & 0 & 0 & 1 & -1 & 1 & 0 & 0 & 0 & 0 & 0 & 0 & 0 & -1 & 1 & -1\\ 
$ p_{156} $ & 0 & 0 & 0 & 0 & 0 & 0 & 0 & 0 & 0 & 0 & 0 & 0 & 0 & 0 & 0 & 0\\ 
$ p_{157} $ & 0 & 0 & 0 & 0 & -1 & 0 & 0 & 1 & 0 & 0 & 1 & 1 & 0 & -1 & 0 & -1\\ 
$ p_{167} $ & 0 & 0 & 1 & 0 & -1 & 0 & 0 & 0 & 0 & 1 & 0 & 0 & 1 & -1 & 0 & -1\\ 
$ p_{234} $ & 0 & 0 & 0 & 0 & 0 & 0 & 0 & 0 & 0 & 0 & 0 & 0 & 0 & 0 & 0 & 0\\ 
$ p_{235} $ & 0 & 0 & 0 & 0 & 0 & 0 & 0 & 0 & 0 & 0 & 0 & 0 & 0 & 0 & 0 & 0\\ 
$ p_{236} $ & 0 & 0 & 0 & 0 & 0 & 0 & 0 & 0 & 0 & 0 & 0 & 0 & 0 & 0 & 0 & 0\\ 
$ p_{237} $ & 0 & 0 & 0 & 0 & -1 & 0 & 0 & 0 & 0 & 1 & 0 & 1 & 0 & -1 & 1 & -1\\ 
$ p_{245} $ & 0 & 0 & 0 & 0 & 0 & 0 & 0 & 0 & 0 & 0 & 0 & 0 & 0 & 0 & 0 & 0\\ 
$ p_{246} $ & 0 & 0 & 0 & 0 & 0 & 0 & 0 & 0 & 0 & 0 & 0 & 0 & 0 & 0 & 0 & 0\\ 
$ p_{247} $ & 0 & 0 & 0 & 0 & -1 & 0 & 0 & 1 & 1 & 0 & 0 & 0 & 1 & -1 & 0 & -1\\ 
$ p_{256} $ & 0 & 0 & 0 & 0 & 0 & 0 & 0 & 0 & 0 & 0 & 0 & 0 & 0 & 0 & 0 & 0\\ 
$ p_{257} $ & 0 & 0 & 1 & 0 & -1 & 1 & 1 & 0 & 0 & 0 & 0 & 0 & 0 & -1 & 0 & -1\\ 
$ p_{267} $ & 0 & 1 & 0 & 1 & -1 & 0 & 0 & 0 & 0 & 0 & 1 & 0 & 0 & -1 & 0 & -1\\ 
$ p_{345} $ & 0 & 0 & 0 & 0 & 0 & 0 & 0 & 0 & 0 & 0 & 0 & 0 & 0 & 0 & 0 & 0\\ 
$ p_{346} $ & 0 & 0 & 0 & 0 & 0 & 0 & 0 & 0 & 0 & 0 & 0 & 0 & 0 & 0 & 0 & 0\\ 
$ p_{347} $ & 0 & 0 & 1 & 0 & -1 & 0 & 0 & 0 & 0 & 0 & 1 & 0 & 0 & -1 & 0 & 0\\ 
$ p_{356} $ & 0 & 0 & 0 & 0 & 0 & 0 & 0 & 0 & 0 & 0 & 0 & 0 & 0 & 0 & 0 & 0\\ 
$ p_{357} $ & 0 & 0 & 0 & 1 & -1 & 0 & 0 & 0 & 0 & 0 & 0 & 0 & 1 & 0 & 0 & -1\\ 
$ p_{367} $ & 0 & 0 & 0 & 0 & 0 & 1 & 0 & 1 & 0 & 0 & 0 & 0 & 0 & -1 & 0 & -1\\ 
$ p_{456} $ & 0 & 0 & 0 & 0 & 0 & 0 & 0 & 0 & 0 & 0 & 0 & 0 & 0 & 0 & 0 & 0\\ 
$ p_{457} $ & 0 & 1 & 0 & 0 & 0 & 0 & 0 & 0 & 0 & 1 & 0 & 0 & 0 & -1 & 0 & -1\\ 
$ p_{467} $ & 0 & 0 & 0 & 0 & -1 & 0 & 1 & 0 & 0 & 0 & 0 & 1 & 0 & 0 & 0 & -1\\ 
$ p_{567} $ & 0 & 0 & 0 & 0 & -1 & 0 & 0 & 0 & 1 & 0 & 0 & 0 & 0 & -1 & 1 & 0\\ \hline 
 $m_{\hat{q}}$ & 1\,092 & 12 & 12 & 12 & 12 & 12 & 12 & 12 & 12 & 12 & 12 & 12 & 12 & 12 & 12 & 12
\end{tabular}
\caption{Columns represent the $16$ numerically obtained tropical critical points of the log-likelihood function in the soft limit for $X(3,7)$. The last row represents their multiplicities.}
\label{tab:37}
\end{table}

Recall that the action of $(\CC^*)^7$ on the Pl\"ucker coordinates
 tropicalizes to an additive action of $\RR^7$ on $\RR^{35}$. Modulo this $\RR^7$-action,
  each column of Table \ref{tab:37}, except the first one,  can be represented by a $\{0,1\}$ vector 
  that has precisely three nonzero entries.  For instance, the second column has its entries $1$
  in the rows $p_{137}, p_{267}, p_{457}$, so it identifies the divisor
  $S_{(13),(26),(45)}$ in the stratification of $X(3,6)$.
  This corresponds to a codimension $3$ matroid stratum for $k=3,m=7$. In this manner,
 one automatically learns the $ 15$ strata of type I in Table \ref{table:strata}. On average, the regression error \eqref{eq:regerr} for all $1\,272$ paths was $\sim 0.0045$, the largest one being $\sim 0.162$. The computation took no more than a couple of minutes. 
\end{example}

Extending  Example \ref{ex:trop37} to a more challenging scenario, we applied Algorithm \ref{algo:computetropicalcriticalpoints} to each of the $5 \,211 \,816$ solutions to the maximum likelihood equations \eqref{eq:critpt48} for $k=4,m=8$. This led us to the multiplicities \eqref{eq:48Multiplicities} and the derivation of Theorem* \ref{conj:48} as discussed in Section~\ref{section6}. 

 \appendix \section{Finite field methods \hfill (Appendix by Thomas Lam)} \label{appendixA}

In this appendix we present a proof of Theorem \ref{thm:numbersk=3} 
that is based on point counting over finite fields and the Weil conjectures. 
This proof predates the work reported in this article.
We compute the Euler characteristic of $X(3,m) = {\rm Gr}(3,m)^\circ/(\CC^*)^m$ for small
values of  $m$.

\begin{theorem} \label{thm:appA}
The Euler characteristics of the configuration spaces $X(3,6), X(3,7), X(3,8)$ and $X(3,9)$ are 
$\,26,\, 1\,272, \,188\,112$ and $74\,570\,400$ respectively.
\end{theorem}
\begin{proof}
The variety  $X(k,m)$ can be defined over a finite field $\mathbb{F}_q$. In sufficiently large characteristic, $X(k,m)_{\mathbb{F}_q}$ is smooth. The number of points in $X(3,6)_{\mathbb{F}_q}$,  $X(3,7)_{\mathbb{F}_q}$, $X(3,8)_{\mathbb{F}_q}$, $X(3,9)_{\mathbb{F}_q}$ was computed by Glynn \cite{Gly}, and Iampolskaia, Skorobogatov, Sorokin \cite{ISS}, and can be found in the paper of Skorobogatov \cite[\S 5]{Sko}. For $q = p^n$ with $p >3$, they are 
\begin{align}\label{eq:main}
\begin{split}
\#X(3,6)_{\mathbb{F}_q} &\,=\,(q-2)(q-3)(q^2-9q+21) \\
\#X(3,7)_{\mathbb{F}_q}  &\,=\, (q-3)(q-5)(q^4 -20q^3 + 148q^2 - 468q + 498) \\
\#X(3,8)_{\mathbb{F}_q}  &\,=\, (q- 5)(q^7-43q^6 + 788q^5 - 7937q^4 + 47097q^3 - 162834q^2
 \\ & \qquad + 299280q - 222960) 
+ 840b(q) \\
\#X(3,9)_{\mathbb{F}_q}   &\,=\, q^{10} - 75q^9 + 2530q^8 - 50466q^7 + 657739q^6 - 5835825q^5
+ 35563770q^4 \\& \qquad - 146288034q^3 + 386490120q^2 - 588513120q 
+389442480 \\&\qquad +840(9q^2 - 243q + 1684)b(q) + 30240(9d(q) + 2e(q)).
\end{split}
\end{align}
We have omitted the terms involving $a(q)$ and $c(q)$ in \cite{Sko} which vanish when $p > 3$.   The functions $b(q),d(q),e(q)$ count the number of solutions of certain quadratic equations in $\mathbb{F}_q$:
\begin{align*}
b(q) &\,\,=\,\, \#\{x \in \mathbb{F}_q ~|~ x^2+x+1 = 0\}, \\
d(q) &\,\,=\,\, \#\{x \in \mathbb{F}_q ~|~ x^2+x-1 = 0\}, \\
e(q) &\,\,=\,\, \#\{x \in \mathbb{F}_q ~|~ x^2+1 = 0\}.
\end{align*}

We shall show that the Euler characteristic of $X(k,m)$ is obtained by setting $b(q)=d(q)=e(q) = 2$ and then $q = 1$ in \eqref{eq:main}. Note that $b(q)=d(q)=e(q) = 2$ holds if
\begin{equation}\label{eq:cong}
p  \equiv 1 \mod 3, \qquad p \equiv \pm 1 \mod 5, \qquad p \equiv 1 \mod 4.
\end{equation}
By Dirichlet's Theorem, such primes exist.  Fix $q = p^n$ such that $X(k,m)$ is smooth in characteristic $p$, and such that \eqref{eq:cong} holds.
Consider the zeta function 
$$
Z(q)\,\,:=\,\, \exp\left( \sum_{n=1}^\infty \frac{\#X(k,m)_{\mathbb{F}_{q^n}}}{n} q^{-ns} \right).
$$
It follows from the Grothendieck-Lefschetz trace formula \cite{Del} that $Z(q)$ is a rational function in the variable $q^{-s}$. Moreover, it can be written in the form $Z = P/Q$, where $$
\deg(P) \,=\, \dim H^{{\rm odd}}_c \bigl(X(k,m)_{\mathbb{F}_q}, \bar\QQ_\ell \bigr), 
\qquad \deg(Q) \,=\, \dim H^{{\rm even}}_c \bigl(X(k,m)_{\mathbb{F}_q}, \bar\QQ_\ell \bigr)
$$
are the total dimensions of odd (resp., even) compactly supported \'etale ($\ell$-adic) cohomology.  

By a standard argument in algebraic geometry
(see for example \cite[Theorem 20.5 and Theorem 21.1]{Mil}), the \'etale cohomology of $X(k,m)_{\mathbb{F}_q}$ is isomorphic to the \'etale cohomology of the complex algebraic variety $X(k,m)$, which is in turn isomorphic to the singular cohomology of the complex manifold $X(k,m)_{\CC}$.  Since $X(k,m)$ is smooth and even-dimensional as a real manifold, its compactly supported Euler characteristic is the same as its Euler characteristic.  We conclude that the Euler characteristic is 
$\, \chi(X(k,m)) \,=\, -\deg(Z(q))$.

On the other hand, $Z(q)$ is easy to write down when $\#X(\mathbb{F}_{q})$ is a polynomial in $q$,
as is the case in (\ref{eq:main}).
  Each term $+q^a$ (resp.~$-q^a$) in $\#X(k,m)_{\mathbb{F}_q}$ contributes a linear factor to the denominator (resp. numerator) of $Z(q)$.  Thus $-\deg(Z(q))$ is equal to the evaluation of the polynomial $\#X(k,m)_{\mathbb{F}_q}$ at $q = 1$.  
  By substituting $q=1$ into (\ref{eq:main}) along with 
  $b(q)=d(q)=e(q) = 2$, one obtains the positive integers stated  in Theorem \ref{thm:appA}.
  \end{proof}

 \begin{remark} 
 We point out that each of the numbers  $26, 1\,272, 188\,112, 74\,570\,400$ has one large prime factor, $13, 53,3\,919, 10\,357$. This seems unrelated to the geometry in Section~\ref{section5}.
 \end{remark}

\begin{footnotesize}

 \bigskip \bigskip

\noindent Daniele Agostini, 
Eberhard Karls Universit\"at T\"ubingen, 
\hfill  {\tt daniele.agostini@uni-tuebingen.de}

\noindent Taylor Brysiewicz, Western University,
\hfill  {\tt tbrysiew@uwo.ca}

\noindent Claudia Fevola,  MPI-MiS Leipzig,
\hfill  {\tt claudia.fevola@mis.mpg.de}

\noindent Lukas K\"uhne,  Universit\"at Bielefeld,
\hfill  {\tt lkuehne@math.uni-bielefeld.de}

\noindent Bernd Sturmfels,
MPI-MiS Leipzig,
\hfill {\tt bernd@mis.mpg.de}

\noindent Simon Telen,  MPI-MiS Leipzig,
\hfill  {\tt simon.telen@mis.mpg.de}

\medskip

\noindent Thomas Lam, University of Michigan,
\hfill  {\tt tfylam@umich.edu}

\end{footnotesize}


\begin{thebibliography}{4}

\setlength{\itemsep}{-0.4mm}


\bibitem{ZariskiFrames}
M. Barakat, T. Kuhmichel and M. Lange-Hegermann:
\emph{{$\mathtt{ZariskiFrames}$ -- (Co)frames/Locales of Zariski closed/open
		subsets of affine, projective, or toric varieties}}, (2018--2021),
(\url{https://homalg-project.github.io/pkg/ZariskiFrames}).

\bibitem{BK}
M. Barakat, L. K\"uhne:
{\em Computing the nonfree locus of the moduli space of arrangements and {T}erao's freeness conjecture},
{\tt arXiv:2112.13065}, to appear in Mathematics of Computation.

\bibitem{BokStu}
J.~Bokowski and B.~Sturmfels:
{\em Computational Synthetic Geometry}, Lecture Notes in Math.~1355, Springer, Berlin (1989).

\bibitem{BB}
M.~M.~Bayer and K.A.~Brandt:
{\em Discriminantal arrangements, fiber polytopes and formality}, Journal of Algebraic Combinatorics,~6(3), 229--246 (1997).

\bibitem{BRT}
P.~Breiding, K.~Rose and S.~Timme:
{\em Certifying zeros of polynomial systems using interval arithmetic}, {\tt arXiv:2011.05000}.

\bibitem{HCjl}
P.~Breiding and S.~Timme:
{\em HomotopyContinuation.jl: A Package for Homotopy Continuation in Julia}, Math.~Software -- ICMS 2018, 458--465, Springer International Publishing (2018).

\bibitem{BHK} 
T.~Brysiewicz, H.~Eble and L.~K\"uhne: 
{\em Enumerating chambers of hyperplane arrangements with symmetry}, {\tt arXiv:2105.14542}.

\bibitem{CEGM}
F.~Cachazo, N.~Early, A.~Guevara and S.~Mizera:
{\em Scattering equations: from projective spaces to 	tropical Grassmannians},
J.~High Energy Phys.~{\bf 39} (2019).

\bibitem{CHY}
F.~Cachazo, S.~He and E.~Y. Yuan:
{\em Scattering equations and Kawai-Lewellen-Tye orthogonality},
Physical Review D {\bf 90} (2014) 065001.

\bibitem{CUZ} F.~Cachazo, B.~Umbert and Y.~Zhang:
{\em Singular solutions in soft limits},
 J.~High Energy Phys.~{\bf 148} (2020).

\bibitem{CHKS}
F.~Catanese, S.~Ho\c{s}ten, A. Khetan and B.~Sturmfels:
{\em The maximum likelihood degree},
 American Journal of Mathematics {\bf 128} (2006) 671--697. 
 
\bibitem{Crapo}
H.~Crapo: {\em The combinatorial theory of structures}, in: {\em Matroid theory} (A.~Recski and L.~Locaśz eds.), Colloq. Math. Soc. János Bolyai {\bf 40}, 107--213, North- Hokkand, Amsterdam-New York, 1985.

\bibitem{Del} P.~Deligne: {\em Weil's conjecture}, II. Publ. Math. IHES {\bf 52} (1980) 137--252.

\bibitem{Falk} M.~Falk: {\em A note on discriminantal arrangements}, Proc.~Amer.~Math.~Soc.~{\bf 122} (1994) 1221--1227.

\bibitem{Gly} D.~Glynn: {\em Rings of geometries, II}, J. Comb. Theory A {\bf 49} (1988) 26--66.

\bibitem{GR} E.~Gross and J.~Rodriguez:
{\em Maximum likelihood geometry in the presence of data zeros}, ISSAC 2014 --  39th 
Internat.~Symposium on Symbolic and Algebraic Computation, 232--239, ACM, New York (2014). 

\bibitem{Huh13} J.~Huh: {\em The maximum likelihood degree of a very affine variety},
Compos.~Math.~{\bf 149} (2013) 1245--1266.

\bibitem{HS} J.~Huh and B.~Sturmfels: {\em Likelihood geometry}, Combinatorial Algebraic Geometry, 
Lecture Notes in Mathematics 2108, Springer Verlag,  63--117, (2014).

\bibitem{ISS} A.~Iampolskaia, A.~Skorobogatov and E.~Sorokin:
{\em Formula for the number of [9,3] MDS codes over finite fields},
IEEE Trans. Info. Theory {\bf 41} (1996) 1667--1671.

\bibitem{Koizumi} H.~Koizumi, Y.~Numata and A.~Takemura: {\em On intersection lattices of hyperplane arrangements generated by generic points}, Annals of Combinatorics {\bf 16} (2012) 789--813.

\bibitem{alcove}
M. Leuner: \emph{$\mathtt{alcove}$ -- algebraic combinatorics package for
	$\mathsf{GAP}$}, 2013--2020, (\url{https://github.com/martin-leuner/alcove}).

\bibitem{MS} D.~Maclagan and B.~Sturmfels: 
{\em Introduction to Tropical Geometry},
Graduate Studies in Mathematics, Vol 161, American Mathematical Society, 2015. 

\bibitem{MMIB_DB}
Y. Matsumoto, S. Moriyama, H. Imai and D. Bremner:
\emph{Database of matroids}, 2012,
(\url{http://www-imai.is.s.u-tokyo.ac.jp/~ymatsu/matroid/index.html}).

\bibitem{MMIB}
Y. Matsumoto, S. Moriyama, H. Imai and D. Bremner:
\emph{Matroid enumeration for incidence geometry}, Discrete Comput. Geom.
\textbf{47} (2012) 17--43.

\bibitem{Mil} J.S.~Milne: {\em Lectures on \'Etale Cohomology}, 202 pages (2013),
available at {\tt{www.jmilne.org/math/}}.

\bibitem{OS} L.~Oeding and S.~V.~Sam: {\em Equations for the fifth secant variety of Segre products of projective spaces}, Exp. Math. \textbf{25} (2016) 94--99.

\bibitem{OT}
P.~Orlik and H.~Terao: {\em Arrangements of Hyperplanes}, Grundlehren der Mathematischen Wissenschaften,
vol 200, Springer-Verlag, Berlin (1992).

\bibitem{Ryb}
G.~Rybnikov: {\em On the fundamental group of the complement of a complex
hyperplane arrangement},  Functional Analysis and its Applications~{\bf 45} (2011) 137--148.

\bibitem{SY}
S.~Settepanella and S.~Yamagata: {\em A linear condition for non-very generic discriminantal arrangements}, (2022),  {\tt arXiv:2205.04664}.
 
\bibitem{Sko}  A.~N.~Skorobogatov: {\em On the number of representations of matroids
over finite fields}, Designs, Codes and Cryptography {\bf 9} (1996) 215--226.

\bibitem{Stanley}
R.~P.~Stanley: {\em An introduction to hyperplane arrangements},  Geometric combinatorics, 
volume 13 of IAS/Park City Math. Ser.,  389--496. Amer.~Math.~Soc., Providence, RI (2007).


\bibitem{PS} B.~Sturmfels and L.~Pachter:
{\em Algebraic Statistics for Computational Biology},
Cambridge University Press (2005).

\bibitem{ST} B.~Sturmfels and S.~Telen:
{\em Likelihood equations and scattering amplitudes},
Algebraic Statistics, {\bf 12.2}, (2021) 167--186.

\bibitem{VS} R.~van der Veer and A-L.~Sattelberger:
{\em Maximum likelihood estimation from a tropical 
	and a Bernstein-Sato perspective}, International Mathematics Research Notices, (2022), rnac016, \url{https://doi.org/10.1093/imrn/rnac016}.

	
\bibitem{Zas75}
T.~Zaslavsky:
{\em Facing up to arrangements: face-count formulas for partitions of
space by hyperplanes}, {\em Mem. Amer. Math. Soc.} {\bf 1} (1975), no.~154.
\end{thebibliography}
\end{document}